\documentstyle[11pt,amssymb,amsmath,amscd,amsbsy,amsfonts,theorem,hyperref,accents,fontenc,bbm]{article}

\input xy
\xyoption{all}

\newcommand{\zz}{{\Bbb Z}}

\newcommand{\nn}{{\Bbb N}}
\newcommand{\cc}{{\Bbb C}}
\newcommand{\rr}{{\Bbb R}}

\newcommand{\pp}{{\Bbb P}}
\newcommand{\aaa}{{\Bbb A}}
\newcommand{\ff}{{\Bbb F}}

\newcommand{\op}[1]{\operatorname{#1}}
\newcommand{\ddim}{\operatorname{dim}}
\newcommand{\ddeg}{\operatorname{deg}}
\newcommand{\kker}{\operatorname{Ker}}

\newcommand{\spec}{\operatorname{Spec}}

\newcommand{\Hom}{\operatorname{Hom}}
\newcommand{\End}{\op{End}}

\newcommand{\ffi}{\varphi}

\newcommand{\eps}{\varepsilon}

\newcommand{\la}{\langle}
\newcommand{\ra}{\rangle}

\newcommand{\row}{\rightarrow}

\newcommand{\low}{\leftarrow}
\newcommand{\lrow}{\longrightarrow}
\renewcommand{\leq}{\leqslant}
\renewcommand{\geq}{\geqslant}

\newcommand{\nichego}[1]{}

\newcommand{\ov}[1]{\overline{#1}}
\newcommand{\un}[1]{\underline{#1}}
\newcommand{\wt}[1]{\widetilde{#1}}

\newcommand{\smk}{{\mathbf{Sm}}_k}

\newcommand{\laz}{{\Bbb L}}

\newcommand{\cm}{{\cal M}}

\newcommand{\Ch}{\operatorname{Ch}}

\newcommand{\CH}{\operatorname{CH}}

\newcommand{\dmeffkF}[1]{\op{DM}^{-}_{eff}(k,F)}

\newcommand{\Qed}{\hfill$\square$\smallskip}
\newcommand{\Red}{\hfill$\triangle$\smallskip}

\newenvironment{proof}{\noindent{\it Proof}:}{\vskip 5mm}

\newtheorem{proposition}{Proposition}[section]{\bf}{\it}
\newtheorem{theorem}[proposition]{Theorem}{\bf}{\it}
\newtheorem{lemma}[proposition]{Lemma}{\bf}{\it}
{\bf}{\it}
\newtheorem{definition}[proposition]{Definition}{\bf}{\rm}
\newtheorem{conj}[proposition]{Conjecture}{\bf}{\it}
\newtheorem{observ}[proposition]{Observation}{\bf}{\it}
\newtheorem{example}[proposition]{Example}{\bf}{\rm}
\newtheorem{remark}[proposition]{Remark}{\bf}{\rm}
{\bf}{\rm}
\newtheorem{claim}[proposition]{Claim}{\bf}{\it}
{\bf}{\it}
{\bf}{\it}
{\bf}{\it}

{\bf}{\it}
\newtheorem{corollary}[proposition]{Corollary}{\bf}{\it}
{\bf}{\it}

\begin{document}

\title{Isotropic and numerical equivalence for Chow groups and Morava K-theories
\renewcommand{\thefootnote}{\fnsymbol{footnote}} 
\renewcommand{\thefootnote}{\arabic{footnote}}
}
\author{Alexander Vishik}
\date{}
\maketitle

\begin{abstract}
In this paper we prove the conjecture claiming that, over a flexible field, {\it isotropic Chow groups}
coincide with {\it numerical Chow groups} (with $\ff_p$-coefficients).
This shows that Isotropic Chow motives coincide with Numerical Chow motives. In particular, homs between such objects are finite groups and
$\otimes$ has no zero-divisors.
It provides a large supply of new points for the Balmer spectrum of the Voevodsky motivic category. 
We also prove the Morava K-theory version of the above result, which permits to construct plenty of new points for the Balmer spectrum of the Morel-Voevodsky ${\Bbb{A}}^1$-stable homotopy category.
This substantially improves our understanding of the mentioned spectra whose description is a major open problem.
\end{abstract}

\section{Introduction}

The idea of {\it isotropic realisation} is to provide a
local handy version for an algebro-geometric object, a version whose complexity will be similar to that of a topological object (and so, much simpler). In the first approximation, such local versions should be parametrized by {\it points} of the algebro-geometric world, that is, by all possible
extensions of the base field. A closer look though reveals that
a prime number $p$ should be chosen and only an equivalence class of an extension under certain $p$-equivalence relation matters.
In the motivic case, the idea is to annihilate the motives of
$p$-anisotropic varieties (that is, varieties which have no closed points of degree prime to $p$). This idea appeared first in the work of Bachmann \cite{BQ}, who considered the $\otimes-\triangle$-subcategory generated by motives of quadrics and applied this method successfully to the study of the Picard group of the Voevodsky category. With the latter purpose in mind, the {\it motivic category of a field extension} $DM(E/k;\ff_p)$ was introduced in \cite[Section 4]{PicQ} and studied extensively in \cite{Iso}. The trivial extension $k/k$ case of it
is the {\it isotropic motivic category} $DM(k/k;\ff_p)$ - \cite[Definition 2.4]{Iso}. It is the Verdier localisation of the ``global'' Voevodsky category
$DM(k;\ff_p)$ with respect to the localising subcategory generated by
motives of $p$-anisotropic varieties. 
Combining the global to local localisation functor with the restriction via
field extension we get natural functors from $DM(k)$ to 
$DM(E/E;\ff_p)$, for various field extensions $E/k$. 
The next important observation is that, for the target isotropic
category to be ``handy'', the field $E$ should be chosen carefully. Fortunately, there is a large class of {\it flexible fields} - \cite[Definition 1.1]{Iso}, which are fit for the role. These are purely transcendental extensions of an infinite transcendence degree of some other fields. Any field $E$ can
be embedded into its' {\it flexible closure} $\wt{E}:=E(\pp^{\infty})$ and the respective motivic restriction functor is conservative. We get the family of {\it isotropic realisations}:
$$
\psi_{p,E}:DM(k)\row DM(\wt{E}/\wt{E};\ff_p),
$$
where $p$ is prime and $E/k$ runs over all
extensions of the base field. These functors are tensor-triangulated and so, the kernels
${\frak a}_{p,E}:=\kker(\psi_{p,E})$ are $\otimes-\triangle$-ideals of the Voevodsky category. It appears that some of these
realisations will be equivalent in a sense. Namely, one can introduce an equivalence relation $\stackrel{p}{\sim}$ on the 
set of extensions - see Section \ref{sect-isotr-realis}, such
that the respective kernels ${\frak a}_{p,E}$ coincide if and only if the extensions are equivalent - 
see Theorem \ref{Bal-sp-Voev}.

As in the global case, the isotropic motivic category 
$DM(k/k;\ff_p)$ possesses the natural weight structure in the sense of
Bondarko \cite{Bon} whose heart is the category of 
{\it isotropic Chow motives} $Chow(k/k;\ff_p)$ - see Proposition
\ref{ws-iso-mot}. This consists of direct
summands of the isotropic motives of smooth projective varieties.
Many questions about the isotropic motivic category can be
reduced to questions about the heart. In this article we will be
able to answer various such crucial questions. 

The Homs in the category of isotropic Chow motives are given by
{\it isotropic Chow groups} $\Ch_{iso}^*$. The latter is the quotient of the usual 
Chow groups modulo {\it anisotropic classes}, that is, elements
coming from $p$-anisotropic varieties. 
On $\Ch^*=\CH^*/p$ of a smooth projective variety $X$ there
is the {\it degree pairing} $\Ch^*(X)\times\Ch^*(X)\row\ff_p$.
Moding out the kernel of it, we get the {\it numerical} version
$\Ch^*_{Num}$. By obvious reasons, anisotropic classes are numerically trivial and so, we get the natural surjection
$\Ch^*_{iso}\twoheadrightarrow\Ch^*_{Num}$ of oriented cohomology theories. Many important properties of the isotropic motivic
categories and isotropic realisations depend on the following
conjecture - \cite[Conjecture 4.7]{Iso}.

\begin{conj}
 \label{conj-Ch-Main}
 If $k$ is flexible, then $\Ch^*_{iso}=\Ch^*_{Num}$.
\end{conj}

In \cite{Iso}, this conjecture was proven for varieties of dimension $\leq 5$, for cycles of dimension $\leq 2$, and for divisors. In \cite{IN}, it was extended to varieties of dimension
$\leq 2p$ and cycles of dimension $<p$ (as well as to few other cases). 
Finally, we can prove the entire Conjecture - see Theorem \ref{conj}.

\begin{theorem}
 \label{conj-thm}
 The conjecture \ref{conj-Ch-Main} is true.
\end{theorem}

An immediate consequence of it is that 
{\it isotropic Chow groups} are finite groups - see 
Corollary \ref{fin-is-Chow-gr}. Another one is that
the category of {\it isotropic Chow motives} is equivalent to the category of {\it numerical Chow motives} - see Corollary \ref{is-Chow-num-Chow}. The latter tensor additive category is semi-simple and has no $\otimes$-zero-divisors. The weight complex functor considerations then imply that the zero ideal $(0)$ in
the isotropic motivic category of a flexible field is prime $\otimes-\triangle$-ideal. Hence, the kernels ${\frak a}_{p,E}$
of our isotropic realisations $\psi_{p,E}$ provide points
of the Balmer spectrum of the Voevodsky category - see
Theorem \ref{a-p-E-prime}: 

\begin{theorem}
 \label{Bal-sp-Voev}
 The
 $\otimes-\triangle$-ideal ${\frak a}_{p,E}$
 is prime and so, defines a point of the Balmer spectrum $\op{Spc}(DM(k)^c)$ of the Voevodsky category. 
 Two such points ${\frak a}_{p,E}$ and ${\frak a}_{q,F}$ are equal if
 and only if $p=q$ and $E/k\stackrel{p}{\sim}F/k$.
\end{theorem}

We get many new points of the Balmer spectrum, complementing the 
``classical'' ones (provided by the topological realisation). 
This substantially improves our understanding of the spectrum.
Balmer and Gallauer managed to completely describe
the spectra of the subcategories of Tate and Artin-Tate motives in the
Voevodsky category in \cite{BG-tri} and \cite{Gal-19} for special fields, but very little was known about $\op{Spc}(DM(k)^c)$ itself, in general. 
Our results show, in particular, that this Balmer spectrum is pretty large. For example, the cardinality of $\op{Spc}(DM(\rr)^c)$ is $2^{2^{\aleph_0}}$ - see Example \ref{exa-C-R}.

A similar techniques can be applied to the study of the Balmer
spectrum of Morel-Voevodsky stable homotopy category $SH(k)^c$. Here one needs
to generalise the notion of {\it anisotropy}. In Section 
\ref{sect-isotr-equiv} below the notion of $A$-anisotropy for any
({\it small}) oriented theory $A^*$ is introduced. A smooth projective variety is $A$-anisotropic, if the push-forward
$\pi_*:A_*(X)\row A_*(\op{Spec}(k))$ to the point is zero. In the
case of $A^*=\Ch^*=\CH^*/p$ it gives the good old $p$-anisotropy. Then, in the same way as for Chow groups, one may introduce the {\it isotropic} version $A^*_{iso}$ of the theory,
with the natural surjection to the numerical one $A^*_{Num}$.
The theories of interest are Morava K-theories
$K(m)^*$ and closely related theories $P(m)^*$ - see Section
\ref{subsect-p-typ-theor}. The latter is the {\it free} theory
(in the sense of Levine-Morel - \cite{LM}) obtained from the
Brown-Peterson theory $BP^*$ by change of coefficients
$P(m)^*:=BP^*\otimes_{BP}P(m)$, where $P(m)=BP/I(m)$ and
$I(m)=(v_0,\ldots,v_{m-1})$ is the invariant ideal of Landweber.
In the case $m=\infty$, we get $P(\infty)^*=K(\infty)^*=\Ch^*$.
Theorems \ref{Main} and \ref{Main-Morava} generalise Theorem 
\ref{conj-thm} to the case of an arbitrary $m$:

\begin{theorem}
 \label{thm-general}
 Let $k$ be flexible. Then, for any $1\leq m\leq\infty$,
 we have: $P(m)^*_{iso}=P(m)^*_{Num}$ and $K(m)^*_{iso}=K(m)^*_{Num}$.
\end{theorem}

This is the main result of the paper, which implies all the rest.
The broadening of view from Chow groups to $P(m)$-theories was crucial in the solution of the Conjecture \ref{conj-Ch-Main}. The right set up here is provided by a theory $Q^*$ which is a quotient of the $BP$-theory by any invariant ideal, and this is the generality in which I prove the Main Theorem \ref{thm-general}. In the proof of the original cases in \cite[Theorem 4.11]{Iso} and 
\cite[Theorem 2.3]{IN}, the key step was to represent our numerically trivial cycle by the class of a regular connected subvariety and then annihilate its' Chern classes numerically. This was achieved by a sophisticated surgery-type techniques involving blow-ups, deformations, as well as Steenrod and Adams operations, and was quite non-trivial. Now it is replaced by the claim that any numerically trivial element of $Q^*$ can be lifted to an element of $\Omega^*$, so that all the Landweber-Novikov operations of it are
$Q$-numerically trivial - see Proposition \ref{killing-SLN-Gen}. The latter statement is proven by a short transparent argument involving the multiplicative projector defining the
$BP$-theory.  It allows to reduce to the case of a  numerically trivial class of a regular embedding $[Y\row X]$, such that the normal bundle of it has the same
$Q$-Chern classes as some bundle of the same dimension coming from $X$. Passing to the flag variety, one
may assume that the class of our normal bundle in $K_0$ is a sum of classes of line bundles, all defined on $X$. Finally, one employs
the defomation to the normal cone in a new unusual role. This famous construction, among other things, also makes
a class into a complete intersection (in a situation like ours). Thus, we may reduce to the case of a numerically trivial complete
intersection, which is straightforward - see Proposition \ref{anis-compl-int}.   

Theorem \ref{thm-general}, in particular, shows that the {\it isotropic category of $K(m)$-Chow motives} $Chow^{K(m)}_{iso}(k)$ coincides with its'
numerical counterpart $Chow^{K(m)}_{Num}(k)$, which is semi-simple and has no $\otimes$-zero-divisors - see Theorems \ref{Chow-is-num-Pm} and \ref{Iso-no-zd}, which plays a key role in the study of isotropic realisations of $SH^c(k)$.

To any prime $p$, any $1\leq m\leq\infty$ and a choice
of a field extension $E/k$ we may assign the 
isotropic realisation
$$
\psi_{(p,m),E}:SH(k)^c\row SH_{(p,m)}(\wt{E}/\wt{E})^c,
$$
where the target is an {\it isotropic stable homotopic category}
(over a flexible field). This is done
in our paper with Du \cite{BsMV}.
This category is obtained from $SH(\wt{E})^c$ by moding out
the $\Sigma^{\infty}_{\pp^1}$-spectra of $K(m)$-anisotropic
varieties as well as objects whose $MGL$-motive is
annihilated by some power of $v_m$ (where we set $v_{\infty}=1$) - see Section \ref{sect-BsMV}.

Let ${\frak a}_{(p,m),E}=\kker(\psi_{(p,m),E})$. 
One may introduce the $K(m)$-equivalence relation $\stackrel{(p,m)}{\sim}$ on the set of extensions - see Section \ref{sect-BsMV}.
Theorem \ref{thm-general} provides the crucial ingredient,
which permits to establish that the zero ideal $(0)$ in a
(flexible) isotropic stable homotopic category is prime, and 
hence, the same holds for ${\frak a}_{(p,m),E}$. This is done
in \cite{BsMV} - see Theorem \ref{SH-points}:

\begin{theorem} {\rm (\cite{BsMV})}
 \label{thm-BsMV}
 \begin{itemize}
  \item[$(1)$] ${\frak a}_{(p,m),E}$ provide points of the Balmer spectrum of Morel-Voevodsky category $SH(k)^c$.
  \item[$(2)$] ${\frak a}_{(p,m),E}={\frak a}_{(q,n),F}$ if and only if $p=q$, $m=n$ and $E/k\stackrel{(p,m)}{\sim}F/k$.
  \item[$(3)$] ${\frak a}_{(p,\infty),E}$ is the image of the point ${\frak a}_{p,E}$ under the natural map of spectra
  $\op{Spc}(DM(k)^c)\row\op{Spc}(SH(k)^c)$.
 \end{itemize}
\end{theorem}

Thus, we get plenty of new ``isotropic'' points of the Balmer spectrum of Morel-Voevdosky category. These points are analogous to and complement the ``classical'' Morava-points 
${\frak a}_{(p,m),Top}$ given by the topological realisation.
The structure of $\op{Spc}(SH(k)^c)$ was studied, in particular,
by Balmer and Heller-Ormsby in \cite{BalSSS} and \cite{HO}, who have shown
that this spectrum surjects to the spectra of Grothendieck-Witt and Milnor-Witt rings of $k$. There are few other results, in particular, comparison with equivariant homotopic spectra. But the question of the description of $\op{Spc}(SH(k)^c)$ remained widely open, and the above isotropic points provide by far the largest known 
structured piece of it (note that there is a huge number of equivalence classes of extensions, for a general field).

 \medskip

\noindent
{\bf Acknowledgements:}
The support of the EPSRC standard grant EP/T012625/1 is gratefully acknowledged. I would like to thank the Referee for useful suggestions which improved the article.

\section{Isotropic equivalence}
\label{sect-isotr-equiv}

Everywhere in this article, the ground field $k$ will be of characteristic zero.

Let $A^*$ be an oriented cohomology theory with localisation in the sense of \cite[Definition 2.1]{SU} (which is the standard axioms of Levine-Morel \cite[Definition 1.1.2]{LM} plus the excision axiom $(EXCI)$). We can introduce the notion of $A$-anisotropic varieties.

\begin{definition}
 \label{A-anis-var} Let $X\stackrel{\pi}{\row}\spec(k)$ be a smooth projective variety. We say that $X$ is $A$-anisotropic, if the map
 $\pi_*:A_*(X)\row A_*(\spec(k))=A$ is zero.
\end{definition}

\begin{observ}
 \label{A-anis-var-obsO} Let $X$ be a smooth projective variety over $k$.
 Then it is $A$-anisotropic if and only if the $A^*$-numerical ring 
 $A^*_{Num}(X)$ of $X$ is zero.
\end{observ}

\begin{proof}
 It is sufficient to recall that the numerical pairing
 $A^*(X)\times A^*(X)\stackrel{\la\,,\,\ra}{\lrow}A$ is defined by
 $\la x,y\ra:=\pi_*(x\cdot y)$.
 \Qed
\end{proof}

\begin{observ}
 \label{A-anis-var-obsD} Let $f:X\row Y$ be a morphism of smooth projective varieties. Then if $Y$ is $A$-anisotropic, then so is $X$.
\end{observ}

\begin{proof}
 Since $f$ is proper, $(\pi_X)_*=(\pi_Y)_*\circ f_*$.
 \Qed
\end{proof}

\begin{remark}
 \label{A-anis-var-rem}
 \begin{itemize}
  \item[$(1)$] Since every smooth projective variety contains a point $\spec(E)$ of finite degree $n=[E:k]$ and for the respective projection,
  $\pi_*\circ\pi^*=\cdot n$, it follows that $A$-anisotropic varieties may exist only if the theory $A^*$ is $n$-torsion: $n\cdot A=0$, for some natural $n$.
  \item[$(2)$] When $A^*=\CH^*/n$ is the Chow groups modulo $n$, our notion of $A$-anisotropy coincides with the notion of $n$-anisotropy: $X$ is $\CH^*/n$-anisotropic if and only if the degrees of all closed points of $X$ are divisible by $n$ -
  see \cite[Definition 2.14]{Iso}.
  \Red
 \end{itemize}
\end{remark}

Now we can introduce {\it isotropic equivalence} on $A^*$.

\begin{definition}
 \label{A-anis-elt}
 Let $X$ be a smooth projective variety and $x\in A^*(X)$. We say that
 $x$ is anisotropic, if $x=f_*(y)$ for some $f:Y\row X$ and $A$-anisotropic (smooth projective) $Y$.
\end{definition}

We have:

\begin{proposition}
 \label{A-anis-elt-properties}
 Anisotropic classes form an ideal which is stable under pull-backs and push-forwards, and so, under $A^*$-correspondences.
\end{proposition}

\begin{proof}
 The fact that anisotropic classes are stable under push-forwards is obvious from definition. As for pull-backs, if $x$ belongs to the image
 of $f_*$, where $f:Q\row X$ and $Q$ is $A$-anisotropic, then it also
 belongs to the image of $\pi_*$, where $\pi:X\times Q\row X$ is the
 projection and $X\times Q$ is still anisotropic. But $\pi$ is transversal to any $g:Y\row X$. Thus, $g^*(x)$ belongs to the image of 
 $\wt{\pi}_*$, where $\wt{\pi}:Y\times Q\row Y$, by the axiom $(A2)$ - see \cite[Definition 1.1.2]{LM}, and $Y\times Q$  
 is anisotropic too. Also, clearly, anisotropic classes form an ideal under external products. The remaining properties follow.
 \Qed
\end{proof}

We get the {\it isotropic version} of the theory $A^*$.

\begin{definition}
 \label{iso-theory}
 Let $X$ be smooth projective over $k$. Define 
 $$
 A^*_{iso}(X):=A^*(X)/(\,\text{anisotropic classes}\,).
 $$
\end{definition}

Using \cite[Example 4.1]{Iso} (cf. \cite[Example 4.6]{RNCT}) this can be extended to an oriented cohomology theory (with localisation) on $\smk$ as $A^*_{\Gamma}$, where $\Gamma$ consists of all classes from $A$-anisotropic (smooth projective) varieties.

If $f:Y\row X$ is any morphism of smooth projective varieties and $y\in A^*(Y)$, $x\in A^*(X)$, then
$\la y,f^*(x)\ra=\la f_*(y),x\ra\in A$. Thus, any {\it anisotropic} class
is {\it numerically trivial} (by Observation \ref{A-anis-var-obsO}) and we have a surjection
$A^*_{iso}\twoheadrightarrow A^*_{Num}$ of oriented cohomology theories.

\section{Brown-Peterson theory and Morava K-theories}

\subsection{Multiplicative operations}
\label{subs-three-one}

To any oriented cohomology theory $A^*$ one can assign a formal group
law $(A,F_A)$, where $A=A^*(\spec(k))$ and $F_A$ describes the Chern class of the tensor product of two line bundles in terms of the Chern classes of the factors: $c^A_1(L\otimes M)=F_A(c^A_1(L),c^A_1(M))$.

If $A^*$ and $B^*$ are two oriented cohomology theories on $\smk$, a {\it multiplicative operation} $G:A^*\row B^*$ is a collection of ring homomorphisms $A^*(X)\row B^*(X)$, for all smooth $X/k$, commuting with
pull-back maps. The difference with the {\it morphisms of theories} is
that the above maps don't have to commute with push-forwards. 

To any multiplicative operation $G:A^*\row B^*$ one can assign the morphism 
$$
(\ffi_G,\gamma_G):(A,F_A)\row (B,F_B)
$$ 
of the respective formal group laws, where $\ffi_G:A\row B$ is the action of $G$ on $\spec(k)$ and
$\gamma_G(x)\in B[[x]]x$ describes, how $G$ acts on the Chern classes of line bundles:
$$
G(c_1^A(L))=\gamma_G(c_1^B(L)).
$$
These data satisfies:
$$
\ffi_G(F_A)(\gamma_G(x),\gamma_G(y))=\gamma_G(F_B(x,y)).
$$

An oriented cohomology theory is {\it free}, if it is obtained from algebraic cobordism of Levine-Morel \cite{LM} by change of coefficients:
$A^*(X)=\Omega^*(X)\otimes_{\laz}A$. Such theories are in 1-to-1 correspondence with formal group laws: $A^*\,\,\leftrightarrow\,\,
\laz\row A$.
By \cite[Theorem 6.9]{SU}, the multiplicative operations 
$G:A^*\row B^*$ between free theories are in 1-to-1 correspondence with the morphisms
of the respective FGLs.

We say that a multiplicative operation $G:A^*\row B^*$ is of {\it invertible type}, if the leading coefficient $\gamma'(0)$ is invertible in $B$. Such a power series is then invertible with respect to the  composition of series. We say that $G$ is {\it stable}, if this coefficient is equal to $1$. 

In the case when the source theory is the algebraic cobordism of Levine-Morel $\Omega^*$, a multiplicative operation $G$ of invertible type is completely determined by $\gamma_G$ (due to universality of the 
respective FGL), which can be any power series with an invertible coefficient at $x$.

The structure on the algebraic cobordism is provided by the Landweber-Novikov operations. The {\it total Landweber-Novikov operation}
$$
S_{L-N}^{Tot}: \Omega^*\row\Omega^*[\ov{b}],
$$
where $\ov{b}=b_i,\,i\geq 1$ and $\ddim(b_i)=i$ is given by the homomorphism $(\ffi,\gamma)$ of  formal group laws, where
the ({\it stable}) change of parameter $\gamma(x)=x+b_1x^2+b_2x^3+\ldots$ has formal variables as coefficients - see \cite[Example 4.1.25]{LM}.  
This operation is universal among {\it stable multiplicative operations}.
Namely, any such operation $G$ can be completed uniquely to a commutative square
$$
\xymatrix{
\Omega^* \ar[r]^{S_{L-N}^{Tot}} \ar[d]_{\theta_A}& \Omega^*[\ov{b}] \ar[d]^{\psi}\\
A^*\ar[r]^{G} & B^*
}
$$
where $\theta_A$ is the canonical morphism of theories (reflecting the universality of algebraic cobordism - \cite[Theorem 1.2.6]{LM}) and 
$\psi$ extends a similar canonical morphism $\theta_B$ by sending the formal variables $b_i$ to the coefficients of the power series $\gamma_G(x)\in B[[x]]x$.

Let $G:A^*\row B^*$ be a multiplicative operation of invertible type between {\it free} theories. 
Let $X$ be a smooth projective variety and 
$\op{Td}_G(X)\in B^*(X)$ be the
{\it Todd genus} of $X$ corresponding to the operation $G$ - see \cite[Definition 2.5.2]{P-RR}. In other words,
$$
\op{Td}_G(X)=\prod_j\left(\frac{x}{\gamma_G(x)}\right)(\lambda_j),
$$
where $\lambda_j$ runs over all {\it $B$-roots} of the Tangent bundle $T_X$ (recall, that {\it $B$-roots} are ``formal elements'', elementary symmetric functions of which are $B$-Chern classes of our bundle - see \cite[2.3]{SU} and \cite[4.1.8]{LM}; below we will also call by this name any collection of elements of degree $1$ in $B^*(X)$
satisfying the same property, if one exists).
$\op{Td}_G(X)$ is an invertible element in $B^*(X)$.
Then the Riemann-Roch theorem for multiplicative operations (of invertible type) - Panin \cite[Theorem 2.5.4]{P-RR} claims that the transformation $G':A^*\row B^*$, defined on a given variety $Y$ by $G\cdot\op{Td}_G(Y)$, commutes with (proper) push-forwards.

\subsection{$p$-typical theories}
\label{subsect-p-typ-theor}

The theories we are interested in are related to a choice of a prime number $p$ and can be produced out of the algebraic cobordism $\Omega^*$
of Levine-Morel \cite{LM}. Algebraic cobordism is unversal among all oriented cohomology theories. For $\zz_{(p)}$-localised theories, this role is played by the Brown-Peterson theory $BP^*$ \cite{BP}. This theory is obtained
from $\Omega_{\zz_{(p)}}^*$ by the same multiplicative projector $\rho$
as in topology (see \cite[I.3]{Wi} and \cite{Ad}). Namely, $\rho:\Omega_{\zz_{(p)}}^*\row\Omega_{\zz_{(p)}}^*$ is a multiplicative operation, which satisfies: $\rho\circ\rho=\rho$ and is determined by the property that $\rho([\pp^n])=[\pp^n]$, if $n$ is a power of $p$ minus one, and is zero, otherwise. In particular, (by the formula of Mischenko) the logarithm of this theory is given by:
$$
\log_{BP}(x)=\sum_{i\geq 0}\frac{[\pp^{p^i-1}]}{p^i}x^{p^i}.
$$
By \cite[Proposition 4.9(2)]{SU}, $BP^*$ is obtained from algebraic 
cobordism by change of coefficients: $BP^*(X)=\Omega^*(X)\otimes_{\laz}BP$, that is, it is a {\it free theory}.
Here $\laz\row BP$ is the universal {$p$-typical} formal group law.
A torsion-free FGL over $\zz_{(p)}$ is called {\it $p$-typical}, 
if its logarithm has
only terms of degree a power of $p$, for the general case see
\cite[Definition 2.1.22]{Rav}. We call an oriented cohomology theory $A^*$ over $\zz_{(p)}$ 
{\it $p$-typical} if the respective FGL is.
The coefficient ring of the Brown-Peterson theory has the form $BP=\zz_{(p)}[v_i; {i\geq 1}]$, where
$\ddim(v_i)=p^i-1$. It is traditional to denote $v_0=p$.
The Brown-Peterson theory is universal among {\it $p$-typical} oriented theories. Namely, for any such theory $A^*$, there exists a unique 
morphism of theories $\theta_A: BP^*\row A^*$ - see \cite[4.1.12, A2.1.25]{Rav} and \cite[Theorem 6.9]{SU}. 

The {\it Total Landweber-Novikov operation} 
$$
S^{Tot}_{BP}:BP^*\row BP^*[\ov{t}],
$$
where $\ov{t}=t_i$, $i\geq 1$, $\ddim(t_i)=p^i-1$, is a stable multiplicative operation given
by the morphism $(\ffi,\gamma)$ of formal group laws with 
$\gamma^{-1}(x)=\sum_{i\geq 0}^{BP}t_ix^{p^i}$ and $t_0=1$ (note the formal summation!).
This operation is universal among stable multiplicative operations between 
$p$-typical cohomology theories. That is, any such operation $G$
extends uniquely to a commutative square:
\begin{equation}
\label{BP-ops-univ}
\xymatrix{
BP^* \ar[r]^{S_{BP}^{Tot}} \ar[d]_{\theta_A}& 
BP^*[\ov{t}] \ar[d]^{\psi}\\
A^*\ar[r]^{G} & B^*
}
\end{equation}
where $\theta_A$ is the canonical morphism of theories (classifying $p$-typical theories) and $\psi$ is an extension of $\theta_B$, mapping 
$t_i$ to $d_i\in B$, where $\gamma^{-1}_G(x)=\sum_{i\geq 0}^{B}d_ix^{p^i}$ -
see \cite[Lemma A2.1.26]{Rav}.

We have an algebraic version $K(p,m)$ of Morava K-theory, which is an oriented cohomology theory in the sense of \cite[Definition 2.1]{SU} obtained from $BP$-theory by the change of coefficients: $K(p,m)^*=BP^*\otimes_{BP}K(p,m)$, where $BP\row K(p,m)$ is a
$p$-typical formal group law of level $m$. That is, for some choice of generators of $BP$, $K(p,m)=\ff_p[v_m,v_m^{-1}]$, where $v_i\mapsto 0$, for $i\neq m$. Following the topological tradition, we will also denote this theory simply as $K(m)^*$.

By the result of Landweber \cite[Theorem 2.7]{La73b}, the prime
ideals of $BP$ invariant under Landweber-Novikov operations are exactly the ideals $I(m)=(v_0,v_1,\ldots,v_{m-1})$,
for $0\leq m\leq\infty$. Consider the {\it free} theory 
$P(m)^*$ obtained from $BP^*$ by change of coefficients
$P(m)^*(X)=BP^*(X)\otimes_{BP}P(m)$, where $P(m)=BP/I(m)$. In particular, 
$P(0)^*=BP^*$, $P(1)^*$ is the mod-$p$ version $\ov{BP}^*=BP^*/p$ and
$P(\infty)=\Ch^*=\CH^*/p$ is Chow groups modulo $p$. For $m\geq 1$, these
theories are $p$-torsion. In particular, it makes sense to speak about
{\it isotropic equivalence} on $P(m)^*$.

We will also use the theory $P\{m\}^*=P(m)^*[v_m^{-1}]$ obtained from $P(m)^*$ by inverting $v_m$. There are the following natural
morphisms among our $p$-typical theories:
$$
BP^*\row P(m)^*\row P\{m\}^*\row K(m)^*.
$$

\section{The Main Theorem}

Our main result compares isotropic and numerical properties of the $P(m)^*$ theory. Namely, we will show that, for 
$1\leq m\leq\infty$, the theories $P(m)_{iso}^*$ and 
$P(m)_{Num}^*$ coincide.
This plays the crucial role in applications below, and implies the respective facts about Chow groups and Morava K-theories.  

\subsection{The strategy of the proof}

Let me start by explaining the strategy of the proof. The argument applies to any free theory $Q^*$ obtained from $BP^*$ by moding out a non-zero ideal $J$ of $BP$ invariant under $BP$-Landweber-Novikov operations. 

Let $X$ be a smooth projective variety over a flexible field. We need to show that any {\it numerically trivial} class $u$ in $Q^*(X)$ is {\it $Q$-anisotropic}.

The first step is to lift $u$ to an element of the algebraic cobordism of Levine-Morel $\Omega^*(X)$ (via the natural maps
$\Omega^*(X)\row BP^*(X)\row Q^*(X)$) in such a way that the
action of all Landweber-Novikov operations on it produces elements which are numerically trivial in $Q^*(X)$ (as $u$ itself). This is achieved by applying the multiplicative projector defining the $BP$-theory to an arbitrary lifting and uses the fact that the $BP$-Landweber-Novikov operations respect $Q$-numerical equivalence. This
is done in Proposition \ref{killing-SLN-Gen} below, see also 
Proposition \ref{L-N-Pm-Num-Gen}. Thus, we may assume that
$u$ is represented by the class of a projective map 
$[Y\stackrel{y}{\row}X]$ from some smooth variety $Y$ of dimension $d$. 
The $Q$-numerical triviality of the action of Landweber-Novikov operations on $[Y\stackrel{y}{\row}X]$ means that any polynomial in Chern classes of $T_Y$ is $Q$-numerically trivial on $X$.

The second step is to substitute $u$ by a class of a regular 
embedding $[Y\stackrel{y}{\row}X]$, with the normal bundle {\it $Q$-equivalent} to a pull-back of some element from $K_0(X)$ (where we call two elements of $K_0$ to be {\it $Q$-equivalent}, if they have the same $Q$-Chern classes).
 By twisting $T_Y$ by a sufficiently high $p$-primary power of a very ample line bundle (which is a {\it $Q$-equivalence}), we may assume that $T_Y$ is generated by global sections and defines an embedding $Y\stackrel{f}{\row} Gr(d,N)$ into some Grassmannian. Then 
 $$
 v=[Y\stackrel{(y,f)}{\row}X\times Gr(d,N)]
 $$ 
 is a class of a regular embedding,
whose projection to $X$ is $u$ and which is still $Q$-numerically trivial, since the restriction of $Q^*(X\times Gr(d,N))$ to $Y$ as an algebra over $Q^*(X)$ is generated exactly by the Chern classes of $T_Y$, polynomials in which are $Q$-numerically trivial by Step 1. 
It is sufficient to prove the claim for $v$. 
Moreover, since $T_Y$ is {\it $Q$-equivalent} to the restriction of the tautological vector bundle from the Grassmannian, the normal bundle of our embedding has the 
needed property. 

With a bit of extra work, we may reduce to the case of a regular
embedding $u=[Y\stackrel{y}{\row}X]$, whose normal bundle is
{\it $Q$-equivalent} to the restriction of the sum $[L_1]+\ldots+[L_n]$ of classes of line bundles from $X$, where $n=\op{codim}(u)$. 

The third step is to use the deformation to the normal cone construction and reduce to the case of a complete intersection.
Here we replace $X$ by $W=Bl_{X\times\pp^1}(Y\times 0)$ and $u$ by the pull-back of $-[Y\times 0]$. This class is still $Q$-numerically trivial and it's push-forward to $X$ is $-u$. But, at the same time, it is equal to the product of 1-st $Q$-Chern classes of line bundles
$\xi(\xi+_Qa_1)\cdot\ldots\cdot(\xi+_Qa_n)$, where $\xi=c^Q_1(O(1))$ and $a_i=c_1^Q(L_i)$. 
The key here is the fact that the $Q$-roots of the normal bundle $N_{Y\times 0\subset X\times\pp^1}$ are defined on $X$ and there is a zero among them.
By twisting by a sufficiently high $p$-primary power of an ample line bundle, we may assume that our line bundles are very ample and so, our class is a $Q$-numerically trivial complete intersection. 

Finally, on the fourth step we prove that any numerically trivial
complete intersection is anisotropic. This result holds for any {\it free} theory $A^*$ and is proven in 
Proposition \ref{anis-compl-int}. 
This finishes the proof of the Main Theorem.

\subsection{The proof}
Let $Q=BP/J$, where $J\subset BP$ is a non-zero ideal invariant under Landweber-Novikov operations. Consider
the {\it free} theory $Q^*=BP^*\otimes_{BP}Q$. 
Let $Q^*_{iso}$ and $Q^*_{Num}$ be the {\it isotropic} and the {\it numerical} versions of the theory (the latter one is obtainded from $Q^*$ by moding-out the kernel of the numerical pairing on $Q^*$ - see \cite[Definition 4.3]{Iso}). 
We are going to show that $Q^*_{iso}=Q^*_{Num}$. 
The case where $J$ is the Landweber ideal $I(m)$ will give us the statement for $P(m)$ then.

\begin{proposition}
 \label{L-N-Pm-Num-Gen}
 Let $J\subset BP$ be an invariant ideal, $Q=BP/J$ and $Q^*$ be the free theory given by: $Q^*(X)=BP^*(X)\otimes_{BP}Q$.
 Then the Landweber-Novikov operations descend to $Q^*_{Num}$.
\end{proposition}

\begin{proof} 
Since the ideal $J$ is invariant, the Total $BP$-Landweber-Novikov 
operation descends to the $Q^*$-theory:
$$
S^{Tot}_{Q}: Q^*\row Q^*[\ov{t}]
$$
We need to show that Landweber-Novikov operations preserve
the ideal of numerically-trivial classes.

Our operation can be extended to a multiplicative endomorphism of the $Q^*[\ov{t}]$-theory 
$$
G=\wt{S}_{Q}^{Tot}: Q^*[\ov{t}]\row Q^*[\ov{t}],
$$
by mapping $t_i$ to $t_i$. In other words, for
the respective morphism  $(\ffi_G,\gamma_G)$ of the FGLs,
$\ffi_G$ acts on $Q$ as $\ffi_{S^{Tot}_{Q}}$ and maps $t_i$ to $t_i$,
while $\gamma^{-1}_G(x)=\sum_{i\geq 0}^{Q}t_ix^{p^i}$, with $t_0=1$. 
Recall, that according to \cite[Theorem 6.9]{SU}, multiplicative operations between {\it free} theories are
in 1-to-1 correspondence with the morphisms of the respective formal group laws.
Since the constant term of $\ffi_{S^{Tot}_{Q}}$ is the identity map on 
$Q$, and the terms at higher monomials in $\ov{t}$ decrease the dimension of elements, it follows that $\ffi_G:Q[\ov{t}]\row Q[\ov{t}]$ is invertible. Also, the power series $\gamma_G$ is invertible (with respect to the composition). Then the pair $(\ffi_H,\gamma_H)$ given by: $\ffi_H=\ffi_G^{-1}$ and $\gamma_H(x)=\ffi_G^{-1}(\gamma_G^{-1})(x)$
defines an endomorphism of the FGL $(Q[\ov{t}],F_{Q})$ inverse to
$(\ffi_G,\gamma_G)$. 
Recall, that the composition of morphisms of FGLs is defined as: 
$(\ffi_{H\circ G},\gamma_{H\circ G}(x))=
(\ffi_H\circ\ffi_G,\ffi_H(\gamma_G)(\gamma_H(x)))$.
Hence, the operation $G$ is invertible. In particular, it is surjective. 

Let $X\stackrel{\pi}{\row}\spec(k)$ be a smooth projective variety and 
$\op{Td}_G(X)\in Q^*(X)[\ov{t}]$ be the
{\it Todd genus} of $X$ corresponding to our operation $G$ - see Subsection \ref{subs-three-one} and \cite[Definition 2.5.2]{P-RR}.
This is an invertible element in $Q^*(X)[\ov{t}]$.
Recall, that the Riemann-Roch theorem for multiplicative operations - Panin \cite[Theorem 2.5.4]{P-RR} claims that the transformation $G':Q^*[\ov{t}]\row Q^*[\ov{t}]$ defined on a given variety $Y$ by $G\cdot\op{Td}_G(Y)$ commutes with (proper) push-forwards.

Let $x\in Q^*(X)[\ov{t}]$ be a numerically trivial element.
Since $G=\wt{S}^{Tot}_{Q}$ is surjective and $\op{Td}_G(X)$ is invertible, for any 
$z\in Q^*(X)[\ov{t}]$, there exists $y\in Q^*(X)[\ov{t}]$, such 
that $z=\wt{S}^{Tot}_{Q}(y)\cdot\op{Td}_G(X)$. Then
$$
\pi_*(\wt{S}^{Tot}_{Q}(x)\cdot z)=\pi_*(\wt{S}^{Tot}_{Q}(x)\wt{S}^{Tot}_{Q}(y)\cdot\op{Td}_G(X))=\pi_*(\wt{S}^{Tot}_{Q}(x\cdot y)\cdot\op{Td}_G(X)). 
$$
By the Riemann-Roch theorem \cite[Theorem 2.5.4]{P-RR}, using the fact that $\op{Td}_G(\op{Spec}(k))=1$, this is equal to
$$
\wt{S}^{Tot}_{Q}(\pi_*(x\cdot y))\cdot \op{Td}_G(\op{Spec}(k))=
\wt{S}^{Tot}_{Q}(\pi_*(x\cdot y)). 
$$
So, it is zero, since $x$ was numerically trivial.
Hence, $\wt{S}^{Tot}_{Q}(x)$ is numerically trivial too.
 \Qed
\end{proof}

\begin{remark}
 \label{L-N-Pm-Num-Gen-rem}
 This statement applies, among others, to the theories
 $Q^*=BP^*\otimes_{BP}BP/I(m)^r$, for any $r\in\nn$ and $0\leq m\leq\infty$.
 In particular, to all theories $P(m)$ (including $P(0)^*=BP^*$).
 \Red
\end{remark}

The following result is one of the key steps of the construction.
It permits to substitute the non-trivial and elaborate considerations
of \cite[Theorem 4.11]{Iso} and \cite{IN} by a clean and transparent multiplicative projector argument.

\begin{proposition}
 \label{killing-SLN-Gen}
 Let $J\subset BP$ be a non-zero invariant ideal and $Q^*=BP^*\otimes_{BP}BP/J$.
 Let $X/k$ be a smooth projective variety and $x\in Q^*(X)$ be a 
 numerically trivial class. Then it may be represented by 
 $\wt{x}\in\Omega^*(X)$, such that the projection of $S^{Tot}_{L-N}(\wt{x})$ to $Q^*$ is numerically trivial.
\end{proposition}

\begin{proof}
 Let $\rho:\Omega^*_{\zz_{(p)}}\row\Omega^*_{\zz_{(p)}}$ be the multiplicative projector defining the theory $BP^*$. 
 It corresponds to the endomorphism $(\ffi_{\rho},\gamma_{\rho})$ of the ($\zz_{(p)}$-localised) universal FGL - see \cite[Theorem 6.9]{SU}.  
 It can be 
 decomposed as $\Omega^*_{\zz_{(p)}}\stackrel{\mu}{\row}BP^*\stackrel{\eta}{\row}\Omega^*_{\zz_{(p)}}$, where $\mu$ is the canonical morphism of theories (so, the respective morphism of FGLs has the form
 $(pr,x)$) and $\eta$ is the multiplicative operation corresponding
 to the morphism $(em,\gamma_{\rho})$, where $\ffi_{\rho}$ decomposes as
 $\laz_{\zz_{(p)}}\stackrel{pr}{\row}BP\stackrel{em}{\row}\laz_{\zz_{(p)}}$ - see \cite[the proof of Prop. 4.9(2)]{SU}.
 Moreover $\mu\circ\eta=id_{BP^*}$. Thus, $BP^*$ can be realised as a
 quotient of $\Omega^*_{\zz_{(p)}}$, or as a direct summand of
 the {\it reorientation} (as in Panin-Smirnov \cite{PS}) $(\Omega^{\gamma_{\rho}}_{\zz_{(p)}})^*$ of this theory. 
 Consider the composition:
 $$
 BP^*\stackrel{\eta}{\row}\Omega^*_{\zz_{(p)}}\stackrel{S^{Tot}_{L-N}}{\lrow}\Omega^*_{\zz_{(p)}}[\ov{b}]\stackrel{\mu[\ov{b}]}{\lrow}BP^*[\ov{b}].
 $$
 It is a multiplicative operation $\delta: BP^*\row BP^*[\ov{b}]$ between $p$-typical theories and so, 
 by (\ref{BP-ops-univ}), 
 is a specialization of 
 $S^{Tot}_{BP}:BP^*\row BP^*[\ov{t}]$. That is, there exists a
 morphism of theories $\eps: BP^*[\ov{t}]\row BP^*[\ov{b}]$ (which is
 automatically $BP^*$-linear), such that
 $\delta=\eps\circ S^{Tot}_{BP}$. In other words, individual $\delta$-operations (coefficients of $\delta$ at particular monomials in 
 $\ov{b}$) are $BP$-linear combinations of the individual 
 $BP$-Landweber-Novikov operations $S^{t^{\vec{r}}}_{BP}$. 
 In turn, it is easy to see (from universality of the cobordism operation $S^{Tot}_{L-N}$) that $S^{Tot}_{BP}$ is a specialization of $\delta$ (and so, these operations are equivalent), but we will not use it.
 
 Let $\ov{x}\in\Omega^*_{\zz_{(p)}}$ be any lifting of $x\in Q^*$,
 and $\wt{x}=\rho(\ov{x})\in\Omega^*_{\zz_{(p)}}$. Then
 $$
\mu[\ov{b}]\circ S^{Tot}_{L-N}(\wt{x})=\mu[\ov{b}]\circ S^{Tot}_{L-N}\circ\eta\circ\mu(\ov{x})=\delta\circ\mu(\ov{x}). 
 $$
We know that the projection of $\mu(\ov{x})$ to $Q^*$ is numerically trivial and that $\delta$ is a specialization of $S^{Tot}_{BP}$. 
 From Proposition \ref{L-N-Pm-Num-Gen}, the projection of $\delta\circ\mu(\ov{x})$ to $Q^*$ is numerically trivial. This means 
 that the projection of $S^{Tot}_{L-N}(\wt{x})$ to $Q^*$ is
 numerically trivial. It remains to recall that $\mu\circ\eta=id_{BP^*}$.
 Hence, the projection of $\wt{x}$ to $Q^*$ 
 is equal to that of $\mu\circ\eta\circ\mu(\ov{x})=\mu(\ov{x})$ and so,
 coincides with $x$.
 Finally, since $J$ is non-zero and invariant, it contains $p^r$, for some $r$ (note that, by the results of Landweber \cite{La73b}, the radical of this ideal is some Landweber ideal $I(m)$, for $m>0$). Then multiplying $\wt{x}$ by an appropriate integer $q\equiv 1\,(mod\,p^r)$, we may assume that the class $\wt{x}$ is integral, and so, represented by some $[Y\stackrel{y}{\row}X]$ with $Y$-smooth projective. 
 \Qed
\end{proof}

\begin{remark}
 \label{killing-SLN-Gen-rem}
 It applies, in particular, to the theories $P(m)$, for
 any $1\leq m\leq\infty$.
 When $J=0$, that is, $Q^*=BP^*$, we may still represent $x$ by
 $\wt{x}\in\Omega_{\zz_{(p)}}^*(X)$ with the result of Landweber-Novikov operations on it $BP$-numerically trivial, but it will not be represented by an integral class, in general.
 \Red
\end{remark}

Finally, we will need the following result showing that numerically 
trivial complete intersections are anisotropic (over a flexible field).

\begin{proposition}
 \label{anis-compl-int}
 Let $A^*$ be a free oriented cohomology theory, $k$ be flexible and 
 $X/k$
 be smooth projective variety. Let $u=\prod_{i=1}^nx_i\in A^*(X)$, where $x_i=c_1^A(L_i)$, for very ample line bundles $L_i$ on $X$. Suppose,
 $u$ is numerically trivial. Then it is anisotropic.
\end{proposition}

\begin{proof}
 Let $Z$ be a smooth projective variety and $L=O(D)$ a very ample line bundle on it. The linear system $|D|$ of effective divisors linearly equivalent to $D$ is parametrized by a projective space $\pp^N$.
 These divisors are hyperplane sections of $X$ in the projective embedding given by $L$.
 Consider the subvariety $Y\subset Z\times\pp^N$ consisting of pairs
 $\{(z,H)|z\in H\}$. We have natural projections 
 $Z\stackrel{\pi}{\low}Y\stackrel{\beta}{\row}\pp^N$, where $\pi$ is a $\pp^{N-1}$-bundle (note that, for a given $z$, the condition $z\in H$ gives a single linear equation on $H$). Let $\xi=c_1^A(O(1))$ on it. 
 Let $\eta$ be the generic point of $\pp^N$, and $Y_{\eta}$ be the 
 generic fiber of the projection $\beta$. It is the {\it generic} 
 representative of the linear system $|D|$ defined over the purely
 transcendental extension $F=k(\pp^N)$ of the base field.
 By the projective bundle axiom (PB) \cite[Def. 1.1.2]{LM} and 
 by excision (EXCI) and \cite[Def. 2.6, Cor. 2.13]{RNCT}, we have surjections:
 $$
 A^*(Z)[\xi]\twoheadrightarrow A^*(Y)\twoheadrightarrow A^*(Y_{\eta}).
 $$
 Here $O(1)$ on $Y$ is the restriction of the bundle $O(1)$ on $\pp^N$.
 Thus, positive powers of $\xi$ are supported in positive codimension on $\pp^N$ and so, restrict to zero in $A^*(Y_{\eta})$. Thus, we obtain
 the surjection $A^*(Z)=A^*(Z_F)\stackrel{j^*}{\twoheadrightarrow} A^*(Y_{\eta})$ for the embedding $j:Y_{\eta}\row Z_F$ of the generic section.
 
 Let $L_i=O(D_i)$ and $\pp^{N_i}=|D_i|$ be the respective ample linear
 systems. Let $\pp=\prod_i\pp^{N_i}$, $\eta$ be the generic point 
 of $\pp$ and $j:W\hookrightarrow X_E$ be the intersection of the generic representatives
 of our linear systems. It is defined over a purely transcendental
 extension $E=k(\pp)$ of $k$.
 Applying the above argument inductively, we obtain that the restriction 
 $A^*(X)=A^*(X_E)\stackrel{j^*}{\twoheadrightarrow}A^*(W)$ is surjective.
 Since $u_E=[W\row X_E]^A$ is numerically trivial, this shows that all
 elements of $A^*(W)$ are numerically trivial.
 Indeed, the (zero) pairing $\la u_E,-\ra: A^*(X_E)\row A$ is just the composition 
 $$
 A^*(X_E)\stackrel{j^*}{\twoheadrightarrow}A^*(W)\stackrel{pr_*}{\row} A^*(\op{Spec}(E))=A.
 $$ 
 This means that $W$ is 
 anisotropic. Thus, $u$ becomes anisotropic over some purely transcendental extension $E=k(\pp)$. Since $k$ is {\it flexible}, that is, 
 $k=k_0(t_1,t_2,\ldots)$ and the variety $X$ and the class $u$ are defined
 on some finite level: over some $k_n=k_0(t_1,\ldots,t_n)$,
 there is an isomorphism of field extensions $k/k_n$ with 
 $E/k_n$, which identifies $X$ with $X_E$ and $u$ with $u_E$ - see \cite[Proposition 1.3]{Iso}.
 It follows that $u$ is anisotropic already over $k$.
 \Qed
\end{proof}

\begin{remark}
 \label{compl-int-rem}
  Note that here we don't have any restrictions on the free theory $A^*$. In particular, we see that if $A^*$ is not torsion, then complete intersections can't be numerically trivial (since there are no non-zero anisotropic classes, in this case). 
 The same result holds for any quotients of free theories.
 \Red
\end{remark}

Now, we are ready to prove our Main Theorem.

\begin{theorem}
 \label{Main}
 Let $k$ be flexible and $1\leq m\leq\infty$.
 Then $P(m)^*_{iso}=P(m)^*_{Num}$. 
\end{theorem}

Again, we will prove a more general version.

\begin{theorem}
 \label{Main-Gen}
 Let $k$ be flexible, $J\subset BP$ be a non-zero invariant ideal
 and $Q^*=BP^*\otimes_{BP}BP/J$.
 Then $Q^*_{iso}=Q^*_{Num}$. 
\end{theorem}

\begin{proof} 
We need to show that all numerically trivial classes in $Q^*$ are
anisotropic.
Suppose, $X/k$ is smooth projective and $x\in Q^*(X)$ is numerically trivial. By Proposition \ref{killing-SLN-Gen}, we may represent $x$ by the
class $u=[Y\stackrel{y}{\row}X]\in\Omega^*(X)$, such that 
$S^{Tot}_{L-N}(u)\in\Omega^*(X)[\ov{b}]$ is numerically trivial in $Q^*(X)$.

The operation $S^{Tot}_{L-N}$ commutes with pull-backs. 
The value of it on $u$ is given by 
$$
S^{Tot}_{L-N}(u)=y_*(\op{Td}(Y))\cdot\op{Td}(X)^{-1}, 
$$
where 
$\op{Td}(Z)=\op{Td}_{S^{Tot}_{L-N}}(Z)$ is equal to $\prod_j\left(\frac{x}{\gamma(x)}\right)(\lambda_j)$, where $\lambda_j$ runs over $\Omega^*$-roots of the
tangent bundle $T_Z$ of $Z$, and $\gamma=\gamma_{S^{Tot}_{L-N}}$ is the ``generic'' stable power series $x+b_1x^2+b_2x^3+\ldots$.
 - see Section \ref{subs-three-one}, \cite[Theorem 2.5.3]{P-RR} and \cite[Example 4.1.25]{LM}.
There is the homological version $S^{L-N}_{Tot}$ of it, which on $X$ is equal to
$S^{Tot}_{L-N}\cdot\op{Td}(X)$.  The value of it on $u$ is equal to 
$$
S^{L-N}_{Tot}(u)=y_*(\op{Td}(Y)).
$$
Due to the Riemann-Roch theorem for multiplicative operations - Panin \cite[Theorem 2.5.4]{P-RR},
$S^{L-N}_{Tot}$ commutes with (proper) push-forwards (which is obvious from the above presentation).
Since $S^{Tot}_{L-N}(u)$ is $Q^*$-numerically trivial, so is
$S^{L-N}_{Tot}(u)=S^{Tot}_{L-N}(u)\cdot\op{Td}(X)$. 

The individual homological Landweber-Novikov operations 
$S^{L-N}_{b^{\vec{r}}}(u)$ are just ($y_*$ of the) polynomials in various Chern classes
$c_{s}^{\Omega}(-T_Y)$ of the minus tangent bundle of $Y$. More precisely, $S^{L-N}_{b^{\vec{r}}}(u)$ is the coefficient
at the monomial $b^{\vec{r}}$ in $y_*(\op{Td}(Y))$. 
Note that, this way, we get all possible polynomials in the
Chern classes of $-T_Y$, or which is the same, all possible polynomials in the Chern classes of $T_Y$.
Since all these classes
are $Q^*$-numerically trivial on $X$, arbitrary polynomials in the Chern
classes of the tangent bundle of $Y$ are $Q^*$-numerically 
trivial on $X$.

Suppose, $\ddim(Y)=d$.
Let $L$ be a very ample line bundle on $Y$. Then, for sufficiently large
$k$, $T_Y\otimes L^{p^k}$ is generated by global sections and 
defines a regular embedding $f:Y\row Gr(d,N)$ into the Grassmanian of $d$-dimensional subspaces of some vector space, such that 
$T_Y\otimes L^{p^k}=f^*(Tav)$ is the pull-back of the
tautological vector bundle.

Thus, our class $u\in\Omega^*(X)$ is the push-forward of the class of the
regular embedding $g=(y,f)$: 
$$
v=[Y\stackrel{g}{\lrow}X\times Gr(d,N)]\in\Omega^*(X\times Gr(d,N)).
$$ 
Note that $Q^*(X\times Gr(d,N))$
as a $Q^*(X)$-algebra is generated by the Chern classes $c_i^{Q}(Tav)$ (see \cite[Example 14.6.6]{Fu}) and the pairing $\la v^Q,-\ra:Q^*(X\times Gr(d,N))\row Q$ is given by
\begin{equation}
\label{vQ-pairing}
Q^*(X\times Gr(d,N))\stackrel{g^*}{\lrow}Q^*(Y)
\stackrel{pr_*}{\lrow}Q^*(\op{Spec}(k))=Q.
\end{equation}
By Lemma \ref{MT-L0}, for $k$ sufficiently large in comparison to the $\ddim(Y)=d$, the first Chern 
class $c_1^{Q}(L^{p^k})$ is zero (in other words, $L^{p^k}$ is {\it $Q$-equivalent} to the trivial bundle $O$). As {\it $Q$-equivalence} is respected by the $\otimes$ (as Chern classes of the tensor product are expressible in terms of
Chern classes of factors), the Chern classes of the pull-back $f^*(Tav)$, for sufficiently large $k$, are just the Chern classes of $T_Y$. But we know that any polynomial in the latter Chern classes is $Q^*$-numerically trivial on
$X$. Hence, for such $k$, the map (\ref{vQ-pairing})
is zero, and so,
the class $v\in\Omega^*(X\otimes Gr(d,N))$
is $Q^*$-numerically trivial.

Note that the normal bundle of $g$ is $Q$-equivalent to
$g^*([T_{X\times Gr(d,N)}]-[Tav])$. Thus, 
replacing $X$ by $X\times Gr(d,N)$ and $u$ by $v$, we reduce to the case,
where $u=[Y\stackrel{y}{\row}X]$ is a class of a regular embedding and the normal bundle $N_{Y\subset X}$ is
{\it $Q$-equivalent} to the pull-back $y^*([V]-[U])$, for some vector bundles $V$ and $U$ on $X$.  By construction, 
$\ddim(N_{Y\subset X})=\ddim([V]-[U])$. 
Since the normal bundle of $X=\pp_X(O)\row\pp_X(O\oplus U)$
is $U$,
replacing $X$ by $\pp_X(O\oplus U)$ and $Y\stackrel{y}{\row}X$ by $Y\stackrel{y}{\row}X\row\pp_X(O\oplus U)$ (note that $Q^*$-numerical triviality of these elements is equivalent, and the same applies to $Q^*$-anisotropy, as each can be obtained from the other by push-forwards),
we can assume that $N_{Y\subset X}$ is {\it $Q$-equivalent}
to the pull-back $y^*(V)$ of a vector bundle of the same dimension.

Let $\eps: Fl_X(V)\row X$ be the variety of complete flags
on $V$. Replacing $X$ by $Fl_X(V)$ and $Y$ by $\eps^{-1}(Y)$, we may 
assume that $[V]=[L_1]+\ldots+[L_n]\in K_0(X)$, for some line bundles $L_i$. Let $c_1^Q(L_i)=a_i$. Note that anisotropy of $[y]$ is equivalent to the anisotropy of $[\eps^*(y)]$,
since anisotropic classes are stable under pull-backs, push-forwards
and multiplication by any elements by Proposition 
\ref{A-anis-elt-properties} (and $[y]$ can be obtained from
$[\eps^*(y)]$ by the latter two operations - recall, that 
$\eps$ is a consecutive projective bundle). Similarly, $Q^*$-numerical triviality of $[y]$ implies that of $[\eps^*(y)]$.

Let's apply the deformation to the normal cone construction. Consider
$W=Bl_{X\times\pp^1}(Y\times 0)$ with the projection
$\pi:W\row X\times\pp^1$. The class
$\ov{u}=[Y\times 1]=[Y\times 0]\in Q^*(X\times\pp^1)$ is still 
numerically trivial (since $[Y]\in Q^*(X)$ is). Hence, so is
$\pi^*(\ov{u})$. Since the normal bundle $N_{Y\times 0\subset X\times\pp^1}$
is $O\oplus N_{Y\subset X}$, which has $Q$-roots $0,a_1,\ldots,a_n$, all defined on $X$, by \cite[Proposition 5.27]{so2}, 
$$
-\pi^*(\ov{u})=\xi(\xi+_{Q}a_1)\cdot\ldots\cdot(\xi+_{Q}a_n)-
0\cdot a_1\cdot\ldots\cdot a_n=\xi(\xi+_{Q}a_1)\cdot\ldots\cdot(\xi+_{Q}a_n),
$$
where $\xi=c_1^{Q}(O(1))$ - see Lemma \ref{MT-L1} for details. Thus, it is a product of first Chern classes of
certain line bundles.
Note that any line bundle $L$ is a ratio of some very ample line bundles $K\otimes M^{-1}=(K\otimes M^{p^k-1})\otimes M^{-p^k}$.
Thus, tensoring by $M^{p^k}$, for a very ample line bundle $M$ on $W$, which is a {\it $Q$-equivalence}
for sufficiently large $k$, and using the fact that the tensor product of very ample line bundles is very ample, 
we may assume that the $Q^*$-divisor
classes $\xi,(\xi+_{Q}a_1),\ldots,(\xi+_{Q}a_n)$ are represented
by very ample divisors on $W$. Hence, $-\pi^*(\ov{u})$ is a $Q^*$-numerically trivial complete intersection. 
By Proposition \ref{anis-compl-int}, $\pi^*(\ov{u})$ is anisotropic.
But the push-forward of this class to $Q^*(X)$ is equal to $u$ (recall that $\ov{u}=[Y\times 1]$ and this cycle is away from the center of the blow-up).
Thus, $u$ is anisotropic. The Theorem is proven.
 \phantom{a}\hspace{5mm}\Qed
\end{proof}

\begin{lemma}
 \label{MT-L0}
 Let $L$ be a line bundle on $Y$. Then,
 for $k$ sufficiently large in comparison to the $\ddim(Y)=d$,
 the $Q$-Chern class $c_1^Q(L^{p^k})$ is zero.
\end{lemma}

\begin{proof}
 Indeed, $c_1^Q(L^{p^k})=[p^k]_Q(c_1^Q(L))$, but
 since $Q^*$ is $p^r$-torsion, for some $r$ (as $J$ is invariant and non-zero), the $Q$-formal multiplication $[p^r]_Q(x)$ has no linear term, and so, $[p^{rm}]_Q(x)$
has no terms of degree less than $2^m$, so it is sufficient to take $k>r\cdot\op{log}_2(d)$ (recall, that the first Chern class is supported in positive co-dimension, and so, the $(d+1)$-st power of it is zero). 
 \Qed
\end{proof}

\begin{lemma}
 \label{MT-L1}
 Let $B\stackrel{g}{\row}A$ be a regular embedding of co-dimension $d$, where $N_{B\subset A}=N$ is $Q$-equivalent to $g^*(M)$ for some vector bundle $M$ on $A$.
 Let $\pi:Bl_A(B)\row A$ be the blow-up of $A$ at $B$ and
 $\wt{M}=\pi^*(M)$. Then 
 $\pi^*([B])=c_d^Q(\wt{M})-c_d^Q(\wt{M}\otimes O(1))\in Q^*(Bl_A(B))$.
\end{lemma}

\begin{proof}
 Applying \cite[Proposition 5.27]{so2} to the cartesian square and blow-up diagram:
 $$
 \xymatrix{
 A & B \ar[l]_{g}& & A & Bl_A(B) \ar[l]_(0.53){\pi} \\
 B \ar[u]^{g}& B \ar@{=}[l] \ar@{=}[u] & & B \ar[u]^{g} & \pp_B(N) \ar[u]^{j} \ar[l]_(0.53){\eps}
 },
 $$
 taking into account that (in the notations of loc. cit.) $\cm=N$, we get:
 $$
 0=\pi^*g_*+j_*\left(\frac{c_d^Q(N\otimes O(1))-c_d^Q(N)}{c_1^Q(O(-1))}\cdot\eps^*\right).
 $$
Since the Chern classes of $N$ are $j^*$ of the Chern classes of $\wt{M}$,  
using the projection formula and the fact that $j_*(1)=c_1^Q(O(-1))$, we obtain:
 $$
 \pi^*([B])=c_d^Q(\wt{M})-c_d^Q(\wt{M}\otimes O(1)).
 $$
 \Qed
\end{proof}

\begin{remark}
 \label{Main-Gen-rem}
 Since free oriented cohomology theories are stable under purely transcendental field extensions, the above Theorem shows that if
 $k$ is an arbitrary field (of characteristic zero) and 
 $\wt{k}=k(t_1,t_2,\ldots)$ is its {\it flexible closure}, then,
 for any smooth projective $X/k$ and any theory $Q^*$ as above, 
 $$
 Q^*_{Num}(X)=Q^*_{Num}(X_{\wt{k}})=Q^*_{iso}(X_{\wt{k}}).
 $$
\end{remark}

For $m=\infty$, $P(\infty)^*=\Ch^*=\CH^*/p$ is Chow groups modulo $p$, and we obtain:

\begin{theorem}
 \label{conj}
 Let $k$ be flexible. Then isotropic Chow groups $\Ch^*_{k/k}$
 coincide with numerical Chow groups $\Ch^*_{Num}$.
\end{theorem}
 
This settles \cite[Conjecture 4.7]{Iso}.

\begin{corollary}
 \label{fin-is-Chow-gr}
 Let $k$ be flexible. Then isotropic Chow groups $\Ch^*_{k/k}(X)$ of varieties are finite groups.
\end{corollary}

\begin{proof}
 Let $X$ be a smooth projective variety over $k$. It is defined over
 some finitely generated subfield $l\subset k$, that is, $X=\bar{X}_k$,
 for some $\bar{X}/l$. Moreover, any algebraic cycle on $X$ and any rational (mod $p$) equivalence among such cycles
 is defined over some finitely generated subfield $l\subset F\subset k$.
 Thus, $\Ch^*(X)$ is the colimit of $\Ch^*(\bar{X}_F)$, where $F$ runs over all such intermediate finitely generated subfields.
 Since $char(k)=0$, any such $F$ has a complex embedding:
 $F\row\cc$, defining a map $\Ch^*(\bar{X}_F)\row H^{2*}(\bar{X}(\cc);\ff_p)$. Since $H^*(\bar{X}(\cc);\ff_p)=H^*_{et}(\bar{X}_{\cc};\ff_p)$ and the
 latter group is equal to $H^*_{et}(\bar{X}_{\bar{F}};\ff_p)$ by the smooth
 base change theorem \cite[Chapter VI, Cor. 4.3]{Miln}, the above map
 is just the restriction $\Ch^*(\bar{X}_F)\row H^{2*}_{et}(\bar{X}_{\bar{F}};\ff_p)=H^{2*}_{et}(\bar{X}_{\bar{l}};\ff_p)$. In particular, it doesn't depend on the choice of the complex embedding. The same smooth base change theorem shows that for 
 any embedding $F\subset E$ of our finitely generated subfields, the maps
 agree and provide the topological realisation map
 $\Ch^*(X)\row H^{2*}_{et}(\bar{X}_{\bar{l}};\ff_p)$. 
 
 Since numerical pairing is defined on the level of the topological realisation, the numerical Chow groups 
 $\Ch^*_{Num}(X)$ are sub-quotients of the singular cohomology $H^*(\bar{X}(\cc);\ff_p)$ of the topological realisation of $\bar{X}$.
 The latter groups are finite. Due to Theorem \ref{conj}, the same
 is true for isotropic Chow groups (as long as $k$ is flexible).
 \Qed
\end{proof}

Theorem \ref{Main} permits to identify the isotropic and numerical versions for the {\it Morava K-theory} as well.
For this we will need the following result.

\begin{proposition}
 \label{Pm-Km-num-triv}
 Any numerically trivial $K(m)$-class can be lifted to a numerically trivial $P\{m\}$-class.
\end{proposition}

\begin{proof}
 Let $CK(m)^*$ be the free theory with the coefficient ring
 $CK(m)=\ff_p[v_m]$ (the {\it connective Morava K-theory}).
 Denote as $y\mapsto\ov{y}$ the standard projection $P(m)^*\row CK(m)^*$. Let $\pi:X\row\op{Spec}(k)$ be a smooth projective variety.
 Multiplying by some power of $v_m$ (an invertible element), we may assume that our numerically trivial class is lifted to 
 $\ov{x}\in CK(m)^*(X)$ which is the restriction of some
 $x\in P(m)^*(X)$. I claim that:
 
 \begin{claim}
 \label{PmKm-nt-claim}
 Modulo multiplication by some power of $v_m$, $x$ can be chosen in such a way that
 $\ov{S_{P(m)}^{Tot}(x)}\in CK(m)^*[\ov{t}](X)$ is numerically trivial. 
 \end{claim}

 \begin{proof}
 Let $S_{P(m)}^{\ov{t}^{\ov{r}}}=S_{P(m)}^{\ov{r}}$ be the individual $P(m)$-Landweber-Novikov operation (the coefficient of $S_{P(m)}^{Tot}$ at the respective monomial).
 This operation has degree $|\ov{r}|=\sum_i ir_i$ (equal to that of the monomial). 
 We will show by a decreasing induction on $d$ that we may choose $x$ so that, for any $\ov{r}$ and $y\in P(m)^*(X)$ with $|\ov{r}|+\op{codim}(y)\geq d$,
 we have $\pi_*(\ov{S_{P(m)}^{\ov{r}}(x)}\cdot\ov{y})=0$.
 Clearly, for a given choice of $x$, the statement is true for $d>\ddim(x)$, since the dimension of the element
 $\ov{S_{P(m)}^{\ov{r}}(x)}\cdot\ov{y}$ is negative, in this case. On the other hand, the case $d=0$ gives the claim above, because the map $P(m)^*\twoheadrightarrow CK(m)^*$
 is surjective and $P(m)^*(X)$ is generated as a $P(m)$-module by elements of non-negative codimension (since it is so for $\Omega^*$). 
 
 Suppose, the lifting $x$ satisfies the statement for $d+1$. Consider the set 
 $$
 R=R(x)=\{\ov{r}\,|\, \exists y,\,\,\text{such that}\,\,|\ov{r}|+\op{codim}(y)=d\,\,\text{and}\,\,\pi_*(\ov{S_{P(m)}^{\ov{r}}(x)}\cdot\ov{y})\neq 0\}. 
 $$
 We say that $\ov{s}\leq\ov{r}$, if $s_i\leq r_i$, for every $i$
 (in other words, the monomial $\ov{t}^{\ov{s}}$ is a divisor of 
 $\ov{t}^{\ov{r}}$).
 Let 
 $$
 S=S(x)=\{\ov{r}|\,\,\text{such that}\,\,|\ov{r}|\leq d\,\,\text{and}\,\,\forall\ov{s}\leq\ov{r},\,\ov{s}\not\in R(x)\}.
 $$
 Let $\ov{r}\in R$ be a {\it minimal} element of $R$ in the sense of the mentioned partial order. By definition of $R$, there exists $y\in P(m)^*(X)$, such that 
 $|\ov{r}|+\op{codim}(y)=d$ and 
 $\pi_*(\ov{S_{P(m)}^{\ov{r}}(x)}\cdot\ov{y})=\lambda\cdot v_m^a$,
 where $\lambda\in\ff_p$ is invertible. 
 Let $u:=\pi_*(x\cdot y)\in P(m)$.
 Note that since $x$ satisfies the statement for $d+1$
 and by the definition of $R$, we have:
 
 \noindent
 {\bf Fact A}: {\it For any
 $\ov{s}$ and any $z\in P(m)^*(X)$, $\pi_*(\ov{S_{P(m)}^{\ov{s}}(x)}\cdot\ov{z})\in CK(m)$ is divisible by $v_m^a$. Moreover, if $\ov{s}\not\in R(x)$, then it is divisible by $v_m^{a+1}$.}
 
 Indeed, the dimension of any such non-zero element
 must be $\geq$, respectively $>\ddim(x)-d=\ddim(v_m^a)$.
 Since $S_{P(m)}^{Tot}$ is multiplicative and $S^{P(m)}_{Tot}$, which on $X$ is equal to $S_{P(m)}^{Tot}\cdot\op{Td}(X)$, commutes with push-forwards,
 by the Riemann-Roch theorem (for multiplicative operations) -
 Panin \cite[Theorem 2.5.4]{P-RR}, for any $\ov{q}$, we have:
 \begin{equation}
 \label{PmKm-eq}
 \ov{S_{P(m)}^{\ov{q}}(u)}=
 \ov{S^{P(m)}_{\ov{q}}(u)}=
 \sum_{\ov{q}_1+\ov{q}_2+\ov{q}_3=\ov{q}}\pi_*(\ov{\op{Td}(X)^{\ov{q}_1}}\cdot\ov{S_{P(m)}^{\ov{q}_2}(x)}\cdot\ov{S_{P(m)}^{\ov{q}_3}(y)}) 
 \end{equation}
 (here $\op{Td}(X)^{\ov{q}_1}$ is the $\ov{q}_1$-component of $\op{Td}(X)=\op{Td}_{S^{Tot}_{P(m)}}(X)$).
 Combining it with the Fact A, we get:
 
 \noindent
 {\bf Fact B}: {\it $\ov{S_{P(m)}^{\ov{q}}(u)}$ is divisible by $v_m^a$, for any $\ov{q}$. Moreover, if $\ov{q}\in S(x)$, then it is divisible by $v_m^{a+1}$.}
 
 For the second claim, we just neet to observe that if $\ov{q}\in S$, then $\ov{q}_2\not\in R$.
 
 Consider
 $$
 x':=v_m^a\cdot x-\lambda^{-1}\cdot u\cdot S_{P(m)}^{\ov{r}}(x).
 $$
 Note that $\ov{u}=0$ (as $\ov{x}$ is $CK(m)$-numerically trivial)
 and so, $\ov{x'}=\ov{x}\cdot v_m^a$.
 I claim that $x'$ still satisfies the statement for $d+1$, and moreover, $S(x')\supset S(x)\cup\{\ov{r}\}$.
 
 Let $|\ov{s}|+\op{codim}(z)>d$, respectively, 
 $|\ov{s}|+\op{codim}(z)=d$ and $\ov{s}\in S(x)$.
 Since $S_{P(m)}^{Tot}$ is a multiplicative operation and
 $S_{P(m)}^{Tot}(v_m)=v_m$ (note that $v_i$, for $i<m$, are
 zero in $P(m)$), we have: 
 $
 \pi_*(\ov{S_{P(m)}^{\ov{s}}(v_m^a\cdot x)}\cdot\ov{z})=
 v_m^a\cdot\pi_*(\ov{S_{P(m)}^{\ov{s}}(x)}\cdot\ov{z})=0, 
 $
 as
 $x$ satisfies the statement for $d+1$, respectively, 
 $\ov{s}\not\in R(x)$. At the same time,
 $$
 \pi_*(\ov{S_{P(m)}^{\ov{s}}(u\cdot S_{P(m)}^{\ov{r}}(x))}\cdot\ov{z})=\sum_{\ov{s}_1+\ov{s}_2=\ov{s}}
 \pi_*(\ov{S_{P(m)}^{\ov{s}_1}(u)}\cdot
 \ov{S_{P(m)}^{\ov{s}_2}S_{P(m)}^{\ov{r}}(x)}\cdot\ov{z})
 $$
 is equal to
 $v_m^b\cdot\sum_{\ov{s}_1+\ov{s}_2=\ov{s}}
 \pi_*(\ov{S_{P(m)}^{\ov{s}_2}S_{P(m)}^{\ov{r}}(x)}\cdot\ov{w}_{\ov{s}_1})$, for some $\ov{w}_{\ov{s}_1}\in CK(m)^*(X)$,
 where $b=a$, respectively, $b=a+1$, because $\ov{S_{P(m)}^{\ov{s}_1}(u)}$ is divisible by $v_m^b$, by Fact B (note that $\ov{s}_1\in S(x)$, if $\ov{s}$ does). Since
 $S_{P(m)}^{\ov{s}_2}S_{P(m)}^{\ov{r}}$ is a linear combination of Landweber-Novikov operations, by Fact A, our expression is divisible by $v_m^{a+b}$ and so, is zero (since the dimension of the result is less, respectively, equal to that of $v_m^{2a}$). 
 Thus, $\pi_*(\ov{S_{P(m)}^{\ov{s}}(x')}\cdot\ov{z})=0\in CK(m)$.
 Moreover, in the second case,
 the same holds for
 all $\ov{s}'\leq\ov{s}$ (as these elements also belong to $S(x)$).
 Hence, $x'$ satisfies the statement for $d+1$ and $S(x)\subset S(x')$.

 Finally, since $\ov{r}$ is a minimal element of $R(x)$, $\ov{S_{P(m)}^{\ov{r}}(u)}=\lambda\cdot v_m^a$ (as the only non-zero term in (\ref{PmKm-eq}) corresponds to $\ov{q}_2=\ov{r}$), while
$\ov{S_{P(m)}^{\ov{r}'}(u)}$ is divisible by $v_m^{a+1}$, for every other $\ov{r}'\leq\ov{r}$, because the dimension of this element is $>\ddim(v_m^a)$. Hence, for $|\ov{r}|+\op{codim}(z)=d$, we have: 
\begin{equation*}
 \begin{split}
\pi_*(\ov{S_{P(m)}^{\ov{r}}(\lambda^{-1}\cdot u\cdot S_{P(m)}^{\ov{r}}(x))}\cdot\ov{z})&=
\lambda^{-1}\sum_{\ov{r}_1+\ov{r}_2=\ov{r}}
\pi_*(\ov{S_{P(m)}^{\ov{r}_1}(u)\cdot S_{P(m)}^{\ov{r}_2}S_{P(m)}^{\ov{r}}(x)}\cdot\ov{z})\\
&=v_m^a\cdot\pi_*(\ov{S_{P(m)}^{\ov{r}}(x)}\cdot\ov{z})+v_m^{a+1}\sum_{\ov{q}}\pi_*(\ov{S_{P(m)}^{\ov{q}}(x)}\cdot\ov{w}_{\ov{q}}),
\end{split}
\end{equation*}
for some $\ov{w}_{\ov{q}}\in CK(m)^*(X)$. The latter summand is divisible by $v_m^{2a+1}$, by Fact A, and so, is zero, since the dimension of the result is equal to that of $\ddim(v_m^{2a})$. The first summand is equal to
$\pi_*(\ov{S_{P(m)}^{\ov{r}}(v_m^a\cdot x)}\cdot\ov{z})$. Hence, $\pi_*(\ov{S_{P(m)}^{\ov{r}}(x')}\cdot\ov{z})=0$ and
consequently, $\ov{r}\not\in R(x')$ and so, $\ov{r}\in S(x')$ (as all other $\ov{r}'\leq\ov{r}$ were contained in 
$S(x)\subset S(x')$).
Thus, $S(x')$ contains $S(x)\cup\{\ov{r}\}$. 

We managed to increase the cardinality of $S(x)$ while keeping the condition $d+1$. After finitely many such steps $S(x)$ will contain all $\ov{r}$ of degree $\leq d$, which will mean that $x$ satisfies the
condition for $d$. The induction step and the claim is proven. 
\Qed
\end{proof}

So, we may assume that
$\ov{S_{P(m)}^{Tot}(x)}\in CK(m)^*[\ov{t}](X)$ is numerically trivial.
Consider the ideal $J\subset P(m)=BP/I(m)$ generated by $\pi_*(S_{P(m)}^{\ov{r}}(x)\cdot z)$, for all $\ov{r}$
and all $z\in P(m)^*(X)$. It follows from the Riemann-Roch theorem that this ideal is invariant under
$P(m)$-Landweber-Novikov operations. On the other hand, 
we know that the projection of it to $CK(m)$ is zero.
Hence, it doesn't contain powers of $v_m$. By the result of Landweber \cite[Proposition 2.11]{La73b}, the
ideal $J$ must be zero. In particular, $x$ is $P(m)$-numerically trivial. Dividing by an appropriate power of
$v_m$, we obtain a numerically trivial lifting of the original class from $K(m)^*$ to $P\{m\}^*$.
 \Qed
\end{proof}

Let $\pi:X\row\op{Spec}(k)$ be a smooth projective variety and $I(X)=\op{Im}(\pi_*:BP_*(X)\row BP)\subset BP$.
This is an invariant ideal in $BP$ which is finitely generated, since $BP^*(X)$ is generated by elements of non-negative codimension by the result of Levine-Morel \cite[Corollary 1.2.13]{LM}. By the result of Landweber \cite[Theorem 2.7, Proposition 3.4]{La73b}, $\sqrt{I(X)}$ is either $I(n)$, for some $1\leq n<\infty$, or the whole $BP$.
It appears that this ideal contains the information on $P(m)$ and $K(m)$-anisotropy of $X$.

\begin{proposition}
 \label{anis-Pm-Km-IX}
 The following conditions are equivalent:
 \begin{itemize}
  \item[$(1)$] $X$ is $P(m)$-anisotropic; 
  \item[$(2)$] $X$ is $K(m)$-anisotropic;
  \item[$(3)$] $\sqrt{I(X)}=I(n)$, for some $1\leq n\leq m$.
 \end{itemize}
\end{proposition}

\begin{proof}
 As $K(m)^*$ is obtained from $P(m)^*$ by change of coefficients,
 $(1)\Rightarrow(2)$. The condition $(2)$ implies that $I(X)$ doesn't contain powers of $v_m$, which implies $(3)$. Finally,
 $(3)$ implies that $I(X)\subset I(m)\Leftrightarrow (1)$.
 \phantom{a}\hspace{5mm}
 \Qed
\end{proof}

Now we can prove the Morava-analogue of the Main Theorem.

\begin{theorem}
 \label{Main-Morava}
 Let $k$ be flexible and $1\leq m\leq\infty$.
 Then $K(m)^*_{iso}=K(m)^*_{Num}$. 
\end{theorem}

\begin{proof} We know that any $K(m)$-anisotropic class is $K(m)$-numerically trivial. Conversely, by Proposition 
\ref{Pm-Km-num-triv} any $K(m)$-numerically trivial
class can be lifted to a $P\{m\}$-numerically trivial one, which must be $P\{m\}$-anisotropic by Theorem \ref{Main}. Hence, the projection of it to $K(m)^*$ is $K(m)$-anisotropic, since the notions of $K(m)$ and
$P(m)$-anisotropy are equivalent by 
Proposition \ref{anis-Pm-Km-IX}. 
 \Qed
\end{proof}

\section{Applications}

\subsection{Isotropic Chow motives}

Let $A^*$ be an oriented cohomology theory. The category of $A$-correspondences is some linearization of the category of 
smooth projective
varieties. The objects of $Cor^A(k)$ are smooth projective
varieties over $k$, while morphism are defined as $\Hom_{Cor^A(k)}(X,Y)=A^{\ddim(Y)}(X\times Y)$.
Considering the pseudo-abelian (Karoubian) envelope of $Cor^A(k)$
(that is, adding kernels and cokernels of projectors) and formally
inverting the Tate-motive $T\{1\}$ (which can be produced as a direct summand of the motive $M(\pp^1)=T\oplus T\{1\}$ of a projective line), one obtains the category of {\it $A$-Chow motives} $Chow^A(k)$.
This is a tensor additive category with
$M(X)\oplus M(Y):=M(X\coprod Y)$ and $M(X)\otimes M(Y):=M(X\times Y)$.
In the case $A^*=\CH^*$ of Chow groups, we get the classical category of Chow motives $Chow^{\CH}(k)=Chow(k)$ of Grothendieck.

Any morphism of theories $A^*\stackrel{\mu}{\row}B^*$ (a homomorphism
of rings commuting with both pull-backs and push-forwards) defines the 
natural functor $Chow^A(k)\row Chow^B(k)$. 

Let $A^*$ be a free theory and $A^*_{iso}$ and $A^*_{Num}$ be {\it isotropic} and {\it numerical} versions of it. We have a natural
(surjective) morphism of theories: $A^*_{iso}\twoheadrightarrow A^*_{Num}$. Denote the respective Chow-motivic categories as $Chow^A_{iso}(k)$
and $Chow^A_{Num}(k)$, respectively.
Our Main Theorem \ref{Main} and Theorem \ref{Main-Morava} imply:

\begin{theorem}
 \label{Chow-is-num-Pm}
 Let $k$ be flexible and $0<m\leq\infty$. Then
 \begin{itemize}
 \item[$(1)$] The category of isotropic $P(m)$-Chow motives $Chow^{P(m)}_{iso}(k)$
 is equivalent to the category of numerical $P(m)$-Chow motives
 $Chow^{P(m)}_{Num}(k)$.
 \item[$(2)$] The category of isotropic $K(m)$-Chow motives $Chow^{K(m)}_{iso}(k)$
 is equivalent to the category of numerical $K(m)$-Chow motives
 $Chow^{K(m)}_{Num}(k)$.
 \end{itemize}
\end{theorem}

\begin{proof}
 Indeed, from Theorem \ref{Main} it follows that the projection
 $P(m)^*_{iso}\twoheadrightarrow P(m)^*_{Num}$ is, actually, an 
 isomorphism of theories. Similarly, Theorem \ref{Main-Morava}
 implies that the projection
 $K(m)^*_{iso}\twoheadrightarrow K(m)^*_{Num}$ is an 
 isomorphism of theories.
 \Qed
\end{proof}

\begin{proposition}
 \label{Num-no-zd}
 Let $A^*$ be an oriented cohomology theory whose coefficient ring
 $A$ is an integral domain. Then $\otimes$ has no zero-divisors in
 $Chow^A_{Num}(k)$ (on objects or morphisms).
\end{proposition}

\begin{proof}
 Let $\alpha_i\in A^*_{Num}(X_i)$, for $i=1,2$ be non-zero classes.
 Then $\alpha_1\times\alpha_2\in A^*_{Num}(X_1\times X_2)$ is non-zero.
 Indeed, since, $\alpha_i\neq 0$ numerically, there exist 
 $\beta_i\in A^*_{Num}(X_i)$, such that $\la\alpha_i,\beta_i\ra\neq 0\in A$. But then $\la\alpha_1\times\alpha_2,\beta_1\times\beta_2\ra
 =\la\alpha_1,\beta_1\ra\cdot\la\alpha_2,\beta_2\ra\neq 0\in A$,
 since $A$ is an integral domain. Thus, $\alpha_1\times\alpha_2$ is
 numerically non-trivial. Non-zero objects $U_i$, $i=1,2$
 of $Chow^A_{Num}(k)$ are given by non-trivial projectors $\rho_i\in A^*_{Num}(Y_i^{\times 2})$ in the motives of some varities $Y_i$.
 Then $U_1\otimes U_2$ is given by a non-zero projector 
 $\rho_1\times\rho_2$ in $A^*_{Num}((Y_1\times Y_2)^{\times 2})$.
 \Qed
\end{proof}

For $0<m<\infty$, the coefficient ring of $P(m)^*$ is the polynomial
ring $\ff_p[v_m,v_{m+1},\ldots]$ in infinitely many variables over $\ff_p$, while for $m=\infty$, it is $\ff_p$. In any case, it is an integral domain. The same applies to the coefficient ring of Morava K-theory. Thus, combining Theorem \ref{Chow-is-num-Pm} and 
Proposition \ref{Num-no-zd}, we obtain:

\begin{theorem}
 \label{Iso-no-zd}
 Let $k$ be flexible, and $m>0$. Then 
 \begin{itemize}
 \item[$(1)$] The category $Chow^{P(m)}_{iso}(k)$
 of isotropic $P(m)$-Chow motives has no $\otimes$-zero-divisors (on objects or morphisms).
 \item[$(2)$] The category $Chow^{K(m)}_{iso}(k)$
 of isotropic Morava-Chow motives has no $\otimes$-zero-divisors (on objects or morphisms).
 \end{itemize}
 \Qed
\end{theorem}

When $m=\infty$, the category $Chow^{P(\infty)}_{iso}(k)$ is the {\it isotropic Chow motivic category} $Chow(k/k;\ff_p)$ of \cite{Iso}. 
As a particular case of Theorems \ref{Chow-is-num-Pm} and \ref{Iso-no-zd}
we get:

\begin{corollary}
 \label{is-Chow-num-Chow}
 Let $k$ be flexible. Then the category $Chow(k/k;\ff_p)$ of isotropic Chow motives - \cite{Iso} is equivalent to the category 
 $Chow_{Num}(k;\ff_p)$ of numerical
 Chow motives (with $\ff_p$-coefficients).
 In particular, the category $Chow(k/k;\ff_p)$ has no $\otimes$-zero-divisors.
\end{corollary}

Since Homs in the category $Chow(k/k;\ff_p)$ are isotropic Chow groups,
Corollary \ref{fin-is-Chow-gr} gives:

\begin{corollary}
 \label{Ch-mot-fin-mor}
 Let $k$ be flexible. Then Homs in the category $Chow(k/k;\ff_p)$
 are finite groups.
\end{corollary}

The category $Chow(k/k;\ff_p)$ is a full tensor additive subcategory of
the compact part of the triangulated category $DM(k/k;\ff_p)$ of 
{\it isotropic motives} - \cite{Iso}. The above Corollary provides support for the following conjecture.

\begin{conj}
 \label{isotr-mot-fin-mor-conj}
 Let $k$ be flexible. Then Homs between compact objects of the isotropic
 motivic category $DM(k/k;\ff_p)$ are finite groups.
\end{conj}

Another piece of evidence in support of this Conjecture is given by
the calculation of isotropic motivic cohomology of a point (that is, Homs
between Tate objects in $DM(k/k;\ff_p)$). This was done for $p=2$ in 
\cite[Theorem 3.7]{Iso}.

\subsection{Isotropic realisations}
\label{sect-isotr-realis}

In \cite{Iso} the {\it isotropic realisations} of the Voevodsky triangulated category of motives $DM(k)$ were constructed (see also \cite[Introduction]{IN}). Such realisations are parametrized by the
choice of a prime number $p$ and an equivalence class of extensions of the ground field. 

Let $k$ be an arbitrary field (of characteristic zero) and $\Ch^*=\CH^*/p$ be Chow groups modulo $p$. Define the
following partial ordering on the set of field
extensions $E/k$. 
Let $E=\operatornamewithlimits{colim}_{\alpha}E_{\alpha}$ and
$F=\operatornamewithlimits{colim}_{\beta}F_{\beta}$, where
$E_{\alpha}=k(Q_{\alpha})$ and $F_{\beta}=k(P_{\beta})$ are
finitely generated extensions with smooth models $Q_{\alpha}$
and $P_{\beta}$.
We say that $E/k\stackrel{p}{\geq}F/k$, if for any $\beta$,
there exists $\alpha$ and a
correspondence $Q_{\alpha}\rightsquigarrow P_{\beta}$ of degree prime to $p$
(equivalently, the push-forward $\Ch_*(Q_{\alpha}\times P_{\beta})\twoheadrightarrow
\Ch_*(Q_{\alpha})$ is surjective). We say that $E/k\stackrel{p}{\sim}F/k$,
if $E/k\stackrel{p}{\geq}F/k$ and $F/k\stackrel{p}{\geq}E/k$.
For a field $E$, denote as $\wt{E}=E(\pp^{\infty})$ the {\it flexible closure} of $E$. Finally, for a
field $L$, define the {\it isotropic motivic category} $DM(L/L;\ff_p)$
as the Verdier localisation of $DM(L;\ff_p)$ by the localising subcategory
generated by motives of $p$-anisotropic (= $\Ch^*$-anisotropic) 
varieties over $L$ - see \cite[Def. 2.4, Rem. 2.8]{Iso}. We get the family of {\it isotropic realisations} 
$$
\psi_{p,E}: DM(k)\lrow DM(\wt{E}/\wt{E};\ff_p),
$$
where $p$ is an arbitrary prime and $E$ runs over all equivalence classes of the $\stackrel{p}{\sim}$ relation above. Note that our realisations take values
in isotropic motivic categories over flexible fields.

Corollary \ref{is-Chow-num-Chow} permits to reveal the meaning of these functors. They provide the {\it points} of the {\it Balmer spectrum} 
$Spc(DM(k)^c)$ - \cite{Bal} of (the compact part of) the Voevodsky motivic category. 
Recall, that the points of the Balmer spectrum of a tensor triangulated
category are prime $\otimes-\triangle$-ed ideals in it. In our case, the
ideal ${\frak a}_{p,E}=\psi_{p,E}^{-1}(0)$ is the pre-image of the
zero ideal. It is prime, since (the compact part of) the isotropic
(flexible) motivic category $DM(\wt{E}/\wt{E};\ff_p)$ has no zero-divisors by Corollary \ref{is-Chow-num-Chow} as we will prove below.

It was shown by Bondarko \cite{Bon}, that on $DM(k)^c$ there is a 
unique bounded non-degenerate {\it weight structure} whose heart is
the category of Chow motives $Chow(k)$. A similar structure exists
on the category $DM(k/k;\ff_p)^c$ of isotropic motives.

\begin{proposition}
 \label{ws-iso-mot}
  The compact part $DM(k/k;\ff_p)^c$ of isotropic motivic category possesses a bounded non-degenerate weight structure whose heart is the category of isotropic Chow motives $Chow(k/k;\ff_p)$.  
\end{proposition}

\begin{proof}
 Recall from \cite[Proposition 2.5]{Iso} that the compact part 
 $DM(k/k;\ff_p)^c$ of the isotropic motivic category can be realized as a full thick subcategory
 of $DM(k;\ff_p)$ which is an idempotent completion of the subcategory 
 consisting of objects of the form 
 $\wt{\cal X}_{Q}\otimes U$, for all $U\in DM(k;\ff_p)^c$, where
 $\wt{\cal X}_{Q}$ is the reduced motive of the \v{C}ech simplicial
 scheme $\check{C}ech(Q)$, where $Q$ is the disjoint union of all
 smooth $p$-anisotropic varieties over $k$. 
 Here $\wt{\cal X}_{Q}$ is a tensor projector fitting the distinguished triangle
$$
{\cal X}_{Q}\row T\row\wt{\cal X}_{Q}\row {\cal X}_{Q}[1],
$$ 
where ${\cal X}_Q$ is the non-reduced motive of our \v{C}ech simplicial scheme (the complementary projector).
 The subcategory of
 isotropic Chow motives $Chow(k/k;\ff_p)$ is the idempotent completion
 of the subcategory consisting of objects $\wt{\cal X}_{Q}\otimes U$,
 where $U$ is a Chow motive. Since every compact Voevodsky motive is
 an extension of finitely many shifts of Chow motives, the same is
 true about isotropic motives. Following Bondarko, 
 let's define the ${\cal{D}}_{\leq 0}$
 as the idempotent completion of the subcategory of $DM(k/k;\ff_p)$ consisting of all possible extensions (of finitely many) objects of the type $\bar{U}[i]$, where 
 $\bar{U}$ is an isotropic Chow motive and $i\leq 0$, and ${\cal{D}}_{\geq 0}$
 as the idempotent completion of the subcategory of $DM(k/k;\ff_p)$ consisting of all possible extensions (of finitely many) objects of the type $\bar{U}[i]$, where 
 $i\geq 0$. To show that this is indeed a {\it weight structure}, 
 by \cite[Theorem 4.3.2, Proposition 5.2.2]{Bon}, one
 only needs to check that $Hom_{DM(k/k;\ff_p)}(\bar{U},\bar{V}[i])=0$, for any
 isotropic Chow motives $\bar{U}$ and $\bar{V}$ and any $i>0$. 
 
 If 
 $\bar{U}=\wt{\cal X}_{Q}\otimes U$ and 
 $\bar{V}=\wt{\cal X}_{Q}\otimes V$, for (global) Chow motives $U$ and $V$, then 
 $$
Hom_{DM(k/k;\ff_p)}(\bar{U},\bar{V}[i])=
 Hom_{DM(k;\ff_p)}(\wt{\cal X}_{Q}\otimes U,\wt{\cal X}_{Q}\otimes V[i]).
$$
 As $\otimes$-projectors ${\cal X}_{Q}$ and $\wt{\cal X}_{Q}$ define a 
 semi-orthogonal decomposition of the category $DM(k;\ff_p)$ (by 
 \cite[Theorem 2.3.2]{IMQ}), this is equal to:
 $$
Hom_{DM(k;\ff_p)}(T[-i],\wt{\cal X}_{Q}\otimes U^{\vee}\otimes V)=
 H^{{\cal M}}_{-i,0}(\wt{\cal X}_{Q}\otimes U^{\vee}\otimes V;\ff_p).
$$
  But $U^{\vee}\otimes V$ is a direct summand in $M(Y)(j)[2j]$, for some smooth projective variety $Y$ and some $j\in\zz$, while smooth simplicial schemes don't have
 motivic homology $H^{{\cal M}}_{a,b}$, for $a<2b$. Thus, the group in
 question is zero and we obtain the needed weight structure.
 \Qed
\end{proof}

A weight filtration presents an object $X$ as an extension of its graded pieces $X_i$, where $X_i=C_i[i]$, for some objects $C_i$ of the heart.
The latter objects form a {\it weight complex} $t(X)$ - \cite[Def. 2.2.1]{Bon}
$$
\ldots\lrow C_{i+1}\stackrel{d_{i+1}}{\lrow} C_i\stackrel{d_i}{\lrow} C_{i-1}\lrow\ldots
$$
For a bounded weight structure, this complex will be finite.
In contrast to the {\it $t$-structure} filtration, the filtration given
by a weight structure is not canonical. In particular, the weight complex
$t(X)$ is not uniquely defined. But, if for a given $X$, there exists
a weight complex with zero differentials, then 
the respective terms $C_i$ are invariants of $X$ and 
can be considered as {\it weight cohomology} $H^i_w(X)$ of it. In particular, in this case, $X$ is zero if and only if all $H^i_w(X)$s are.
The next statement shows that our case is exactly such. 

\begin{proposition}
\label{wt-zero-dif}
Every object $X$ of $DM(k/k;\ff_p)^c$ possesses a weight filtration producing weight complex with zero differentials.
\end{proposition}

\begin{proof}
 This is a consequence of the semi-simplicity of the category $Chow_{Num}(k,\ff_p)$ - see Proposition \ref{semi-sim-wc}, which follows from the
 following Lemma.
 \Qed
 \end{proof}

 \begin{lemma}
  \label{Cho-num-sem-sim}
  For any morphism $f:U\row V$ in $Chow_{Num}(k,\ff_p)$, there are
  decompositions $U=W\oplus\wt{U}$ and $V=W\oplus\wt{V}$, such that $f=id_W\oplus 0$. In particular, the category
  $Chow_{Num}(k,\ff_p)$ is semi-simple.
 \end{lemma}

 \begin{proof}
 If $f$ is non-zero,
 then the adjoint map $T\stackrel{\ov{f}}{\lrow}U^{\vee}\otimes V$ is non-zero too, which
 means that there exists $g:V\row U$, such that the composition
 $T\stackrel{\ov{f}}\lrow U^{\vee}\otimes V\stackrel{\un{g}}{\lrow} T$
 is the identity. That means that the traces of $\alpha=g\circ f: U\row U$ and $\beta=f\circ g: V\row V$ are non-zero. 
 Let us show that some powers of $\alpha$ and $\beta$ are non-zero projectors. We may assume that $U$ and $V$ are motives of varieties.
 
 Consider the topological realisation functor $TR: Chow(k;\ff_p)\row D^b(\ff_p)$ and denote $H_{Top}(Z):=H^*(TR(Z),\ff_p)$.
 The latter is a finite-dimensional $\ff_p$-vector space.
 We may identify $H_{Top}(U\times U)$ with $\End_{\ff_p}(H_{Top}(U))$. Let $\wt{\alpha}$ and $\wt{\beta}$ be $TR$ of some liftings of $\alpha$ and $\beta$ to non-numerical Chow groups.
 Since the degree pairing is defined on the level of the topological realisation, $\wt{\alpha}$ and $\wt{\beta}$, 
 considered as $\ff_p$-linear transformations, have non-zero traces. Define the filtration:
 $H_{Top}(U)=F_0\supset F_1\supset F_2\supset\ldots$, where
 $F_k=image(\wt{\alpha}^k)$. Since our space is finite-dimensional, there exists $r$, such that the embeddings $F_k\supset F_{k+1}$ are proper, for $k<r$, and are isomorphisms,
 for $k\geq r$. Since the ring $\End_{\ff_p}(H_{Top}(U))$ is finite, there exists $N$, such that $\wt{\alpha}^N$ is a projector. Considered as a linear transformation, it should be
 a projection on $F_r$. Then $tr(\wt{\alpha}^{N+1})$ coincides
 with $tr(\wt{\alpha}|_{F_r})$ which is equal to $tr(\wt{\alpha})$. Hence, $\ddeg(\wt{\alpha}\cdot(\wt{\alpha}^N)^{\vee})=tr(\wt{\alpha}^{N+1})\neq 0$, and so, $\alpha^N$ is numerically non-trivial. Since numerical Chow groups are sub-quotients of topological cohomology (as the degree pairing factors through $TR$), it follows that $a=\alpha^N$ is a non-zero
 projector. Then so is $b=\beta^{2N}$, and maps $f$ and $g'=g\circ\beta^{2N-1}$ identify $(U,a)$ with $(V,b)$.
 Thus, there are decompositions $U=W'\oplus U'$ and $V=W'\oplus V'$
 with non-zero $W'$, s.t. $f=id_{W'}\oplus f'$,
 for some $f':U'\row V'$. Repeating these arguments with $f', f'', ...$ and taking into account that the endomorphism rings of our objects are finite, we get the decomposition as required.
 \Qed
 \end{proof}
 
 \begin{proposition}
  \label{semi-sim-wc}
  Let ${\cal D}$ be a triangulated category with the non-degenerate bounded weight structure, whose heart ${\cal C}$ is semi-simple. Then every object possesses a weight filtration, for which the weight complex has zero differentials.
 \end{proposition}

 \begin{proof}
 Let $X$ be an object of ${\cal D}$. Since our weight structure is bounded, $X$ is supported on some segment $[m,n]$, that is, 
 $X\in{\cal D}_{\geq m}\cap{\cal D}_{\leq n}$. In particular, we have an
 exact triangle
 $$
 X_{>m}[-1]\row X_m\lrow X\lrow X_{>m}.
 $$
 Here the differential $X_{m+1}[-1]\stackrel{d_{m+1}[m]}{\lrow}X_m$
 of the (shifted) weight complex $t(X)$ factors as the composition 
 $X_{m+1}[-1]\row X_{>m}[-1]\row X_m$. Since ${\cal C}$ is semi-simple, we may present it as 
 $(id_W\oplus 0)[m]$, for an appropriate $W\in {\cal C}$.
 Now we can split the triangle $W[m]\stackrel{id}{\row}W[m]\row 0\row W[m+1]$ from the above one, and get another presentation
 $$
 X'_{>m}[-1]\row X'_m\lrow X\lrow X'_{>m}.
 $$
 with the zero differential $C'_{m+1}\stackrel{0}{\row}C'_m$ and
 $X'_{>m}$ supported on $[m+1,n]$. We conclude by induction on the support.
 \Qed
 \end{proof}

The weight cohomology is functorial.

\begin{corollary}
 \label{wt-coh}
 We have a conservative weight cohomology functor 
 $$
 H^{Tot}_w=\oplus_it^i\cdot H^i_w:DM(k/k;\ff_p)^c\row Chow(k/k;\ff_p)[t,t^{-1}].
 $$ 
\end{corollary}

\begin{proof}
 It is straightforward to check that a map $f:X\row Y$ between two
 objects with weight filtrations with zero differentials extends
 uniquely to maps of graded pieces. The conservativity is clear, as any object is an
 extension of graded pieces of it.
 \Qed
\end{proof}

It is just the cohomology of the {\it weight complex functor}
$$
t:DM(k/k;\ff_p)^c\row K^b(Chow(k/k;\ff_p))
$$ 
- \cite{Bon}, which exists (and is triangulated) since, due to Proposition \ref{wt-zero-dif}, the natural projection 
$$
K^b(Chow(k/k;\ff_p))\row K^b_{{\frak w}}(Chow(k/k;\ff_p))
$$ 
is
an equivalence of categories - see \cite[Definition 3.1.6]{Bon}
and \cite[Theorem 3.2.2(II)]{Bon}.   

Moreover, the total weight cohomology functor respects the tensor product.

\begin{theorem} {\rm (Kunneth formula)}
 \label{Kunneth}
 The functor $H^{Tot}_w$ respects the $\otimes$. In other words,
 $$
 H^k_w(X\otimes Y)\cong
 \operatornamewithlimits{\oplus}_{i+j=k}H^i_w(X)\otimes H^j_w(Y).
 $$
\end{theorem}

\begin{proof}
 This follows from the fact that the ``global'' {\it weight complex functor}  is a tensor functor - see \cite[Lemma 20]{BQ}, while the isotropic category is obtained from the global one by
 a tensor projector which commutes with the weight filtration
 and $\otimes$.
 \phantom{a}\hspace{5mm}\Qed
\end{proof}

Applying Corollary \ref{is-Chow-num-Chow}, we get:

\begin{theorem}
 \label{a-p-E-prime}
 For any prime $p$ and any field extension $E/k$, the
 $\otimes-\triangle$-ideal 
 $$
 {\frak a}_{p,E}:=\kker(\psi_{p,E}:DM(k)^c\row DM(\wt{E}/\wt{E};\ff_p)^c)$$
 is prime and so, defines a point of the Balmer spectrum $\op{Spc}(DM(k)^c)$ of the Voevodsky category. 
 
 Two such points ${\frak a}_{p,E}$ and ${\frak a}_{q,F}$ are equal if
 and only if $p=q$ and $E/k\stackrel{p}{\sim}F/k$.
\end{theorem}

\begin{proof}
 This follows from the fact that the total weight cohomology is conservative and satisfies the Kunneth formula, while the $\otimes$
 has no zero-divisors on $Chow(k/k;\ff_p)$ by
 Corollary \ref{is-Chow-num-Chow}.
 
 Since $char({\frak a}_{p,E})=p$, it is an invariant of a point. 
 
 Suppose, $E/k\stackrel{p}{\sim}F/k$. Let $U\in {\frak a}_{p,E}$.
 Let $E=\operatornamewithlimits{colim}_{\alpha}E_{\alpha}$, where
 $E_{\alpha}=k(Q_{\alpha})$ is a finitely generated extension with a smooth model $Q_{\alpha}$. Since $U\in{\frak a}_{p,E}$
 and $U$ is compact,
 there are finitely many $p$-anisotropic varieties $R_i$,
 $i=1,\ldots,n$ over $\wt{E}$, such that $U_{\wt{E}}$ is a direct summand of an extension of their motives. These varieties and all the 
 morphisms involved should be defined over some finitely generated extension $\wt{E}_{\alpha}=\wt{k}(Q_{\alpha})$ of $\wt{k}$. Moreover, since $E/k\stackrel{p}{\sim}F/k$, these varieties 
 remain $p$-anisotropic over $\wt{F}(Q_{\alpha})$. Hence,
 $U_{\wt{F}(Q_{\alpha})}=0\in 
 DM(\wt{F}(Q_{\alpha})/\wt{F}(Q_{\alpha});\ff_p)$. 
 By \cite[Prop. 2.9]{Iso}, then $U_{\wt{F}}=0\in DM(\wt{F}(Q_{\alpha})/\wt{F};\ff_p)$, where the latter is the {\it local motivic category} - see \cite[Def. 2.3]{Iso}. But since $Q_{\alpha}|_{\wt{F}}$ is
 $p$-isotropic, the extensions $\wt{F}(Q_{\alpha})/\wt{F}$ and
 $\wt{F}/\wt{F}$ are $p$-equivalent and so, the respective local categories are identified - see loc. cit.. Thus,
 $U_{\wt{F}}=0\in DM(\wt{F}/\wt{F};\ff_p)$. 
 Hence $U\in {\frak{a}}_{p,F}$. From symmetry, we get that
 ${\frak{a}}_{p,E}={\frak{a}}_{p,F}$, if $E/k\stackrel{p}{\sim}F/k$.
 Conversely, if $E/k\stackrel{p}{\not\sim}F/k$, then either $Q_{\alpha}|_{F}$, for some $\alpha$,
 or $P_{\beta}|_{E}$, for some $\beta$, is $p$-anisotropic. W.l.o.g. may assume the former. Then
 $M(Q_{\alpha})\in {\frak a}_{p,F}$ (as $Q_{\alpha}$ stays anisotropic over $\wt{F}$),
 while $M(Q_{\alpha})\not\in {\frak a}_{p,E}$, as $\psi_{p,E}$ passes through
 $DM(E)$ and a Tate-motive splits from $M(Q_{\alpha})$ there. Hence, the
 points are different.
 \phantom{a}\Qed
\end{proof}

Thus we obtain a large supply of new points of the Balmer spectrum of $DM(k)^c$, which complement the points given by the topological realisation. 

\begin{example}
 \label{exa-C-R}
 If $k$ is algebraically closed. then every variety of finite type over $k$ has $k$-rational points. Hence, for every $p$, there is only one
 $p$-equivalence class of field extensions over $k$.
 
 If $k=\rr$ is real, then there is only one $p$-class, for odd $p$, and $2^{{\frak c}}$ of such classes, for $p=2$, where ${\frak c}=$continuum. Indeed,
 there is a continuum of choices of the $j$-invariant for real curves of genus one without real points. If there is a correspondence of odd degree between two such curves, as there are no real points, the respective
 elliptic curves must be isogenous, and there are only countably many
 elliptic curves isogenous to a given one. Thus we have a set of real points free curves of genus one $\{X_i|i\in I\}$ parametrized by a set $I$ of continuum
cardinality, such that there are no correspondences of odd degree between $X_i$ and $X_j$, for $i\neq j$, and the respective elliptic curves are not isogenous. Moreover, if $Y$ is a product of finitely many $X_i$, where
all $i$ are different from a given $j$, then there is no odd degree correspondence from $Y$ to $X_j$.
Indeed, such a correspondence would give a point of odd degree on $(X_j)_{k(Y)}$, but this variety has also points
of degree $2$ (as we have complex points). Since it is a curve of genus one, by the (classical) Riemann-Roch theorem,
it has a rational point then, that is, there is a rational map $Y\row X_j$, which must be dominant, since $X_j$ has no real points. Then the
respective map of Albanese varieties would have been dominant as well, which would mean that the map of one of the components $Alb(X_i)\row Alb(X_j)$ would be non-zero, which is impossible, since these elliptic curves are not isogenous. For any $J\subset I$, let us introduce the extension
$E_J$, which is the composite of $k(X_j)$, for all $j\in J$. The above considerations show that, for different $J$s, these extensions are not $2$-equivalent.
Hence, we get, at least, $2^{{\frak c}}$ different isotropic points of characteristic $2$. On the other hand, the category $DM(\rr)$ has a continuum of compact objects,
so the cardinality of $\op{Spc}(DM(\rr)^c)$ is no more than $2^{{\frak c}}$. Thus, the number of $2$-equivalence classes of extensions $E/\rr$ is $2^{{\frak c}}$.
Note, that the cardinality of the set of $2$-equivalence classes of {\it finitely generated} extensions $E/\rr$ is only ${\frak c}$. On our way, we established that the cardinality
of  $\op{Spc}(DM(\rr)^c)$ is exactly $2^{{\frak c}}$.
\end{example}

An extension of fields $l\row k$ defines the map of Balmer spectra $\op{Spc}(DM(k)^c)\row\op{Spc}(DM(l)^c)$ mapping the 
point ${\frak{a}}_{p,E/k}$ to ${\frak{a}}_{p.E/l}$. 

\begin{example}
 \label{exa-dif-alg-closed}
 Let $k$ be an arbitrary field (of characteristic zero) and
 $\ov{E}/k$ and $\ov{F}/k$ be two algebraically closed field
 extensions. Then $\ov{E}/k\stackrel{p}{\sim}\ov{F}/k$ and so,
 the points ${\frak{a}}_{p,\ov{E}}$ and ${\frak{a}}_{p,\ov{F}}$
 coincide. Thus, we have the ``algebraically closed isotropic point'' (for every $p$) combining all such extensions. The respective $p$-equivalence class contains, in particular, all extensions $E$ containing $\ov{k}$. Usually, it doesn't contain any finitely-generated representatives. It is the highest class in the sense of the order $\stackrel{p}{\geq}$.
\end{example}

\begin{remark}
 One can show that, at lest, for $p=2$, the points ${\frak a}_{p,E}$ and
 ${\frak a}_{p,F}$, for non-$p$-equivalent finitely generated extensions $E/k$ and $F/k$ of a flexible field $k$, are not comparable in the sense of topology of 
 $\op{Spc}(DM(k)^c)$ (and the same should hold for other primes). Thus,
 the partial order $E/k\stackrel{p}{\geq}F/k$ provides a finer structure
 on the Balmer spectrum, a structure not detected by the topological specialisation of points.
 \Red
\end{remark}

\subsection{Balmer spectrum of Morel-Voevodsky category}
\label{sect-BsMV}

Let $SH(k)$ be the $\aaa^1$-stable homotopy category of Morel-Voevodsky \cite{MV}.
Theorem \ref{Iso-no-zd} permits to construct many new points of 
$\op{Spc}(SH(k)^c)$. This is done in our paper with Du \cite{BsMV}. Here I will just list the main results.

Let $p$ be a prime, $n$ be a natural number, or $\infty$, and $K(p,n)=K(n)$
be the Morava K-theory. Let $Q_{p,n}$ be the disjoint union of all
$K(n)$-anisotropic varieties over $k$, and ${\frak X}_{(p,n)}$
be the $\Sigma^{\infty}_{\pp^1}$-spectrum of the respective \v{C}ech simplicial scheme $\check{C}ech(Q_{p,n})$. It is a $\wedge$-projector, and the complementary projector is $\Upsilon_{(p,n)}=\op{Cone}({\frak X}_{(p,n)}\row{\mathbbm{1}})$. Denote as
$\widehat{SH}_{(p,n)}(k/k)$ the category
$\Upsilon_{(p,n)}\wedge SH(k)$. It is naturally a full subcategory of 
$SH(k)$ equivalent to the Verdier localisation of $SH(k)$ by the localising subcategory generated by $K(n)$-anisotropic varieties. 
The same projector can be applied to the category of modules over any $E_{\infty}$-spectrum. Applying it to
$MGL$-modules we get 
the category $\widehat{MGL}_{(p,n)}-mod$ together with the natural functor 
$$
\widehat{M}^{MGL}:\widehat{SH}_{(p,n)}(k/k)\row\widehat{MGL}_{(p,n)}-mod.
$$
Define the
{\it $K(p,n)$-isotropic stable homotopy category} 
${SH}_{(p,n)}(k/k)$ as the Verdier localisation of $\widehat{SH}_{(p,n)}(k/k)$ by the localising subcategory generated by compact
objects $U$, such that the action of $v_n$ on $\widehat{M}^{MGL}(U)$ is nilpotent, where we set $v_{\infty}=1$. 

We can introduce the partial $K(p,n)$-order on the set of extensions of $k$. Let $E=\operatorname{colim}_{\alpha}E_{\alpha}$ and $F=\operatorname{colim}_{\beta}F_{\beta}$, where $E_{\alpha}=k(Q_{\alpha})$ and $F_{\beta}=k(P_{\beta})$
are finitely-generated extensions with smooth models $Q_{\alpha}$
and $P_{\beta}$.
We say that 
$E/k\stackrel{(p,n)}{\geq}F/k$, if for every $\beta$, there exists $\alpha$, such that $P_{\beta}|_{E_{\alpha}}$ is $K(p,n)$-isotropic ($\Leftrightarrow$
$\pi_*:K(p,n)(Q_{\alpha}\times P_{\beta})\twoheadrightarrow K(p,n)(Q_{\alpha})$ is surjective).
Then $E/k\stackrel{(p,n)}{\sim}F/k$ if $E/k\stackrel{(p,n)}{\geq}F/k$
and $F/k\stackrel{(p,n)}{\geq}E/k$.

We get a family of {\it isotropic realisations}
$$
\psi_{(p,n),E}:SH(k)\row SH_{(p,n)}(\wt{E}/\wt{E}),
$$
where $p$ is a prime number, $1\leq n\leq\infty$, and $E/k$ runs over
equivalence classes of the field extensions under the relation $\stackrel{(p,n)}{\sim}$ above.

For $n=\infty$, these are realisations closely related to $\psi_{p,E}$ of $DM(k)^c$ from the previous section. For not formally-real fields, these coincide with the categories studied by Tanania in \cite{ISHG} and \cite{Tan-COIMC}. 

Our Main Theorem \ref{thm-general} implies that these realisations provide points of the
Balmer spectrum of $SH(k)^c$.
Define ${\frak a}_{(p,n),E}:=\kker(\psi_{(p,n),E})$.

\begin{theorem} {\rm (\cite{BsMV})}
 \label{SH-points}
 \begin{itemize}
 \item[$(1)$]
 The $\otimes-\triangle$-ideals ${\frak a}_{(p,n),E}$ are prime and so,
 define points of $\op{Spc}(SH(k)^c)$. 
 \item[$(2)$]
 Two points ${\frak a}_{(p,n),E}$ and ${\frak a}_{(q,m),F}$ coincide
 if and only if $p=q$, $n=m$ and $E/k\stackrel{(p,n)}{\sim}F/k$.
 \item[$(3)$] 
 The point ${\frak a}_{(p,\infty),E}$ is the image of the point
 ${\frak a}_{p,E}$ under the natural map 
 $\op{Spc}(DM(k)^c)\row\op{Spc}(SH(k)^c)$.
 \end{itemize}
\end{theorem}

These points complement the classical topological points (of positive characteristic).

\bigskip

\begin{itemize}
\item[address:] {\small School of Mathematical Sciences, University of Nottingham, University Park, Nottingham, NG7 2RD, UK}
\item[email:] {\small\ttfamily alexander.vishik@nottingham.ac.uk}
\end{itemize}

\end{document}